\documentclass{svproc}
\usepackage{url}
\usepackage{graphicx}
\usepackage{amsfonts}
\usepackage{color}

\usepackage{comment}
\usepackage{graphicx,epsfig}    

\newcommand{\integer}{\mathbb Z}

\newcommand{\real}{\mathbb R}

\newcommand{\torus}{\mathbb T}

\newcommand{\complex}{\mathbb C}

\newcommand{\Id}{\rm Id}

\newcommand{\C}{\mathcal C}
\newcommand{\D}{\mathcal D}

\newcommand{\PP}{\mathcal P}

\newcommand\beq[1]{ \begin{equation}\label{#1} }
\newcommand{\eeq}{ \end{equation} }

\newcommand\beqa[1]{ \begin{eqnarray} \label{#1}}
\newcommand{\eeqa}{ \end{eqnarray} }
\newcommand{\beqano}{ \begin{eqnarray*} }
\newcommand{\eeqano}{ \end{eqnarray*} }
\newcommand\equ[1]{{\rm (\ref{#1})}}
\newcommand{\eps}{\varepsilon}

\begin{document}
\mainmatter              
\title{KAM theory for some  dissipative systems}
\titlerunning{KAM theory for dissipative systems}  

\author{R. Calleja\inst{1} \and A. Celletti\inst{2} \and R de la Llave\inst{3}}
\authorrunning{R. Calleja et al.} 
%
\tocauthor{Renato Calleja, Alessandra Celletti, and Rafael de la Llave}
\institute{Department of Mathematics and Mechanics, IIMAS, National
  Autonomous University of Mexico (UNAM), Apdo. Postal 20-126,
  C.P. 01000, Mexico D.F., Mexico,\\
\email{calleja@mym.iimas.unam.mx},\\
\texttt{https://mym.iimas.unam.mx/renato/}
\and
Department of Mathematics, University of Roma Tor Vergata, Via
della Ricerca Scientifica 1, 00133 Roma (Italy),\\
\email{celletti@mat.uniroma2.it},\\
\texttt{http://www.mat.uniroma2.it/celletti/}
\and
School of Mathematics,
Georgia Institute of Technology,
686 Cherry St., Atlanta GA. 30332-1160,\\
\email{rafael.delallave@math.gatech.edu},\\
\texttt{https://people.math.gatech.edu/\~rll6/}
}

\maketitle              

\begin{abstract}
Dissipative systems play a very important role in several physical models, most notably in
Celestial Mechanics, where the dissipation drives the motion of natural and artificial satellites,
leading them to migration of orbits, resonant states, etc. Hence the need to develop theories that
ensure the existence of structures such as
invariant tori or periodic orbits and device efficient
computational methods.

The point of view that we adopt is that we are dealing with real
problems and that we will have to use a very wide variety of
methods. From the applications, to numerical studies to rigorous
mathematics. As we will see, all of these methods feed on each
other. The rigorous mathematics leads to efficient algorithms (and
allows us to believe the results), the numerical experiments lead
to deep mathematical conjectures, the applications benefit from
all this tools, and set meaningful goals that prevent from doing
things just because they are easy.   Of course,  the road towards
this lofty goal is not rosy and there are many false starts,
complications, etc. After several years, we can erase the false
starts from the story, but we hope to provide some flavor. Given
the rather wide scope is unavoidable that some arguments have
different standards (rigorous proofs, numerical efficiency,
conjectures). We have strived to make all those very explicit, but
may be it would be hard to keep this present. Of course, similar
programs can be applied to many problems, but in this paper we
will deal with a rather concrete set of problems.

In this work we concentrate
on the existence  of invariant tori for the specific case of dissipative systems known as \sl conformally symplectic \rm systems, which
have the property that they transform the symplectic form into a multiple of itself. To give explicit examples
of conformally symplectic systems, we will present two different models: a discrete system known as
the standard map and a continuous system known as the spin-orbit problem. In both cases we will consider
the conservative and dissipative versions, that will help to highlight the differences between the
symplectic and conformally symplectic dynamics.

For such dissipative systems we will present a KAM theorem in an a-posteriori format:
assume we start with an approximate solution satisfying a suitable
non-degeneracy condition, then we can find a true solution nearby.
The theorem does not assume that the system is close to integrable.

The method  of proof is based on
extending  geometric identities
originally developed in \cite{LlaveGJV05} for the symplectic case.
Besides leading to streamlined proofs of KAM theorem, this method
provides a very efficient algorithm which has been
implemented. Coupling an efficient numerical algorithm
with an a-posteriori theorem, we have a very
efficient way to provide rigorous estimates close to optimal.

Indeed, the method gives a criterion (the Sobolev blow up
criterion)  that allows to compute numerically the breakdown.
We will review this method as well as an extension of J. Greene's method
and present the results in the  conservative and dissipative
standard maps. Computing close to the breakdown, allows
to discover new mathematical phenomena such as the \emph{bundle
  collapse mechanism}.

We will also provide a short survey of the present state of KAM estimates
for the existence of invariant tori in the conservative and dissipative standard maps and spin-orbit problems.
\end{abstract}

\section{Introduction}
Dissipative dynamical systems play a fundamental role in shaping the motions of physical problems.
The role of dissipative forces in Celestial Mechanics is often of less importance with respect to
the conservative forces, which are mainly given by the gravitational attraction between
celestial bodies. Nevertheless dissipative forces are present at any size and time scale and their effect accumulates over time, so that even
if some effects are
negligible in a scale of centuries, they might be dominant in a scale
of a million of years.

A partial list of dissipative forces includes tidal forces, Stokes drag, Poynting-Robertson
effect, Yarkowski/YORP effects, atmospheric drag. These forces act on bodies of different
dimensions, namely planets, satellites, spacecraft, dust particles, and in different epochs
of the Solar system from the dynamics within the interplanetary nebula at the early stage
of formation of the Solar system, to present times. For example,
the effect of the
Earth's atmosphere on the orbital lifetime of artificial satellites,
happens in practical scales of time.
It becomes therefore,
important to understand  invariant structures (e.g., periodic
orbits and invariant tori) in dissipative systems.

The definition of \sl dissipative system \rm is not uniform in
the literature.
Here we will addopt that  a dissipative system has the property
that the phase space volume contracts.
In this work we will be concerned with a special class of dissipative systems known as
\sl conformally symplectic \rm systems. These systems enjoy the property that they
transform the symplectic form into a multiple of itself.

Conformally symplectic systems have appeared in many
applications (e.g. discounted systems, \cite{Bensoussan88}) or have been
studied because they are geometrically natural objects (\cite{Banyaga02}).

For applications to Celestial Mechanics, an important source of conformally symplectic systems is
that of a mechanical system with friction proportional to the velocity.
This is the case of the
so-called spin-orbit problem in Celestial Mechanics (\cite{Celletti90I,Celletti90II}), which will be presented in
Section~\ref{sec:SO}. It describes the motion of an oblate satellite around a central planet,
under some simplifying assumptions like that the orbit of the satellite is Keplerian and
that the spin-axis is perpendicular to the orbit plane. When the satellite is assumed to be rigid,
the problem is conservative, while when the satellite is assumed to be non-rigid, the problem
is affected by a tidal torque. The dissipative part of the spin-orbit problem depends upon two
parameters: the dissipative constant, which is a function of the physical properties
of the satellite, and a drift term, which depends on the (Keplerian) eccentricity of the orbit.
A discrete analogue of the spin-orbit problem is the dissipative standard map (\cite{Chirikov79}).
In Section~\ref{sec:SM} we will review conservative and dissipative
versions of the standard map.

Indeed, the presence of a drift term is fundamental in conformally symplectic systems: while
in the conservative case one can find an invariant torus with fixed frequency by adjusting
the initial conditions,
in the dissipative case it is not possible to just  tune the initial
conditions  to obtain a quasi-periodic
solution of a fixed frequency.
One needs to adjust a drift parameter to find an invariant torus with
preassigned frequency (for some appropriate choice of initial conditions).

We stress that adding a dissipation to a Hamiltonian system is a
very singular perturbation: the Hamiltonian admits quasi-periodic solutions with many frequencies,
while a system with positive dissipation leads to attractors with few quasi-periodic solutions. To obtain attractors with a fixed frequency, one needs
to adjust the drift parameters.

\vskip.1in

The existence of invariant tori is the subject of
the celebrated Kolmogorov--Arnold--Moser (KAM) theory
(\cite{Kolmogorov54,Arnold63a,Moser62}, see also \cite{Llave01c,Wayne96,Poschel01})
which, in its original formulation, proved
the persistence of invariant tori
in nearly--integrable Hamiltonian systems. The theory can be developed under two main assumptions:

- the frequency vector must satisfy a Diophantine condition (to deal with the so-called small divisors problem),

- a non--degeneracy condition must be satisfied (to ensure the solution of the cohomological equations providing
the approximate solutions).

Also, geometric properties of the system play an important role.  Notably,
the original results were developed for Hamiltonian systems, but this
has been greatly extended.

A KAM theory  for non-Hamiltonian systems
with adjustment of parameters was developed in the remarkable and pioneer paper \cite{Moser67},
and later in \cite{BroerHTB90,BroerHS96}. A KAM theory for conformally symplectic systems with adjustment
of parameters was developed in \cite{CallejaCL11} using the so-called
\emph{automatic reducibility} method
introduced in \cite{LlaveGJV05}. The paper
\cite{CallejaCL11} produces  an
a-posteriori result. A-posteriori means that the existence of
an approximate solution, which satisfies an invariance equation up to a small error,
ensures the existence of a true solution of the invariance equation, provided some non-degeneracy
conditions and smallness conditions on parameters are satisfied.

The automatic reducibility proofs
of KAM theorem provide very efficient
and stable algorithms to construct invariant tori in the symplectic (\cite{FiguerasHL17,HaroCFLM16})
and conformally symplectic (\cite{CCL2020,CCGL2020}) case. Examples of concrete (conservative and dissipative) KAM estimates
will be given in Section~\ref{sec:applications} with special reference to the standard map
and the spin-orbit problem. The a-posteriori format guarantees that these
solutions are correct. Indeed, it was proved that the algorithm leads
to a continuation method in parameters that, given enough resources
reaches arbitrarily close the boundary of the set of parameters
for which the solution exists.
In Section~\ref{sec:breakdown} we will review
the results on the empirical study of
the breakdown, which leads to
several very unexpected phenomena.
In \cite{CallejaF11} it was found numerically  that the tori -- which are normally hyperbolic-- break down because the stable bundle becomes close to the tangent,
even if the  stable Lyapunov exponent (which is given by the conformal
symplectic constant) remains away from zero. Moreover, there are remarkable
scaling effects, as shown in Section~\ref{sec:collision}.

\subsection{Other results}
The results presented in these notes are part of  a more
systematic program of providing KAM theorems in
an a-posteriori format with many consequences that, for
completeness, we shortly review below.

$\bullet$ The a-posteriori format, leads automatically to many
regularity results: deducing finitely differentiable results from
analytic ones, bootstrap of regularity, Whitney dependence on
the frequency.
We will not even mention these regularity results, but
we point that in the conformally symplectic case,
we can obtain several rather  striking geometric results.
The conformlally symplectic systems are very rigid.
A classic result that plays a role is
the \emph{paring rule} of Lyapunov exponents. The conformal
geometric structure restricts severely the Lyapunov
exponents that can appear \cite{DettmannM96b,WojtkowskiL98}.

$\bullet$ Rigidity of neighborhoods of
tori.

In \cite{CallejaCL11b} it is shown that the dynamics in a neighborbood
of a Lagrangian torus is conjugate to a rotation and a linear contraction.
In particular,  the only invariant in a neighborhood is the rotation and
all the tori with the same rotation are analytically conjugate in a
neighborbood.

$\bullet$ Greene's method.

An analogue of Greene's method (\cite{Greene79}) to compute the
analyticity breakdown is given in \cite{CCFL14}, which presents
a partial justification of the method. It is proved that when the
invariant attractor exists, then one can predict the eigenvalues
of the periodic orbits approximating the torus for parameter
values close to those of the attractor.

$\bullet$ Whiskered tori.

In \cite{CCLwhiskered,CCLwhiskered2}), one can find a theory of
whiskered tori in conformally symplectic systems.
This theory involves interactions of dynamics and geometry.
The theory allows  -- there are examples -- that the stable
and unstable bundles are trivial, but somewhat surprisingly,
concludes that the center bundles have to be trivial.

$\bullet$ The singular limit of zero dissipation.

We showed in \cite{CCLdomain} that, if one fixes the frequency, one can choose
the drift parameter as a function of the perturbation in a smooth way:
$\mu = \mu(\omega, \varepsilon) \equiv \mu_\varepsilon(\omega) $.  Note however
that $\mu_0 (\omega)=0$, but for
$\varepsilon > 0$,  the function  $\mu_\varepsilon$
is invertible  so that the function $\mu_\varepsilon(\omega)$
is a smooth function with a limit as $\varepsilon \rightarrow 0$.
Nevertheless, the sets of $\omega$ that appear have a complicated
behaviour (devil's staircase). Hence,  the floating frequency
KAM methods, e.g. \cite{Arnold63a,Poschel,ChierchiaG82},
have difficulty dealing with this limit.

One of the advantages of  the a-posteriori theorems is
that they can validate approximate solutions, no matter how
they are obtained. We have already mentioned the validation of
numerical computations. It turns out that one can also validate
formal asymptotic expansions and obtain estimates on
domains of existence of the tori in the singular limit (\cite{CCLdomain}).
This limit has also been studied numerically (\cite{BustamanteC19}),
leading to the conjecture that the Lindstedt series are Gevrey.
A proof of the conjecture is given in \cite{BustamanteL20}.

$\bullet$ Breakdown of the rotational tori.

One of the consequences of the conformal symplectic geometry
is  the ``pairing rule'' for exponents \cite{WojtkowskiL98}. Hence
the tori, which have a dynamics which is a rotation, must
have normal exponents which are $\lambda$.  The tori are normally
hyperbolic attractors. Notice that the loss of hyperbolicity
cannot happen because of the exponents break down.
This leads to the mechanism of bundle collapse that was discovered
in \cite{CallejaF11} and will be
discussed in more detail in Section~\ref{sec:collision}.

\subsection{Organization of the paper}

The work is organized as follows. In Section~\ref{sec:consdiss} we present the conservative and dissipative
standard maps and spin-orbit problems. Conformally symplectic systems and Diophantine vectors are
introduced in Section~\ref{sec:CSDC}. The definition of invariant tori and the statement
of the KAM theorem for conformally symplectic systems is given in Section~\ref{sec:tori}.
Two numerical methods for the computation of the breakdown threshold of invariant attractors
is presented in Section~\ref{sec:breakdown}. The relation between the collision of invariant
bundles and the breakdown of the tori is described in Section~\ref{sec:collision}.
Applications of KAM estimates to the conservative/dissipative standard maps and spin-orbit problems
are briefly recalled in Section~\ref{sec:applications}.

\section{Conservative/dissipative standard maps and spin-orbit problems}\label{sec:consdiss}

In this Section we present two models, a discrete and a continuous one, that will help to
have a qualitative understanding of the main features of conservative and dissipative
systems. The first example is a discrete paradigmatic model, known as the \sl standard map \rm
(see Sections~\ref{sec:SM} and \ref{sec:DSM}).
The continuous example is a physical model, known as the \sl spin-orbit problem, \rm which is closely related to the
standard map (see Section~\ref{sec:SO}). In both cases we present their conservative and dissipative formulations.

\subsection{The conservative standard map}\label{sec:SM}
The standard map is a discrete model introduced by Chirikov in \cite{Chirikov79}, which has been widely studied
to understand several features of dynamical systems,
such as regular motions, chaotic dynamics,
breakdown of invariant tori, existence of periodic orbits, etc. The standard map is a 2-dimensional
discrete system in the variables $(y,x)\in{\mathbb R}\times\torus$,
which is described by the formulas:
\beqa{SM1}
y'&=&y+\varepsilon\  V(x)\nonumber\\
x'&=&x+y'\ ,
\eeqa
where $\varepsilon>0$ is called the \sl perturbing parameter \rm and $V=V(x)$ is an analytic function.

A wide number of articles and books in the literature (see, e.g., \cite{GuckHolmes}, \cite{LichLieberman}) deals with
the classical (Chirikov) standard map (\cite{Chirikov79}) obtained setting $V(x)=\sin x$ in \equ{SM1}.

\vskip.1in

Instead of \equ{SM1} we can use an equivalent
notation and write the standard map assigning an integer
index to each iterate:
\beqa{SM2}
y_{j+1}&=&y_j+\varepsilon V(x_j)\nonumber\\
x_{j+1}&=&x_j+y_{j+1}=x_j+y_j+\varepsilon V(x_j)\qquad {\rm
for}\ \ j\geq 0\ .
\eeqa

\vskip.1in

We can easily verify that the standard map \equ{SM2} satisfies the following properties, that will be useful for
further discussion.

\vskip.1in

$A)$ The standard map is integrable for $\varepsilon=0$. In fact, for $\varepsilon=0$ one gets the formulas:
\beqano
y_{j+1}&=&y_j=y_0\nonumber\\
x_{j+1}&=&x_j+y_{j+1}=x_j+y_j=x_0+jy_0\qquad {\rm for}\ j\geq 0\ ,
\eeqano
which shows that the mapping is integrable, since $y_j$ is constant and $x_j$ increases by $y_0$.
For $\varepsilon\not=0$ but small, the map is nearly-integrable.

\vskip.1in

$B)$ The standard map is conservative, since the determinant of its Jacobian is equal to one:
$$
\det\left(%
\begin{array}{cc}
  {{\partial y'}\over {\partial y}} & {{\partial y'}\over {\partial x}} \\
  {{\partial x'}\over {\partial y}} & {{\partial x'}\over {\partial x}} \\
\end{array}%
\right)=
\det\left(%
\begin{array}{cc}
  1 & \ \ \varepsilon V_x(x_j) \\
  1 & \ \ 1+\varepsilon V_x(x_j) \\
\end{array}%
\right)=1\ .
$$

\vskip.1in

$C)$ The standard map satisfies the so-called twist property, which amounts
to requiring that for a constant $c\in\real$:
$$
\left|{{\partial x'}\over {\partial y}}\right|\geq c>0
$$
for all $(y,x)\in\real\times\torus$. From \equ{SM1} we have that the twist
property is trivially satisfied, since
$$
{{\partial x'}\over {\partial y}}=1\ .
$$

The twist  property is not satisfied when considering a \sl slight \rm modification of \equ{SM1}, yielding a discrete system which is known as
the \sl non-twist \rm standard map (see, e.g., \cite{DelCastillo1,DelCastillo2}). This mapping is described by the equations:
\beqano
y'&=&y+\varepsilon\  V(x)\qquad\qquad y\in \real\ ,\ x\in\torus\nonumber\\
x'&=&x+a(1-y'^2)
\eeqano
with $a\in\real$. In this case, the twist condition is violated along a curve in the $(y,x)$ plane.

Systems violating the twist condition appear in celestial mechanics,
for example in the critical inclination for  the motion
near an oblate planet (\cite{Kyner68}). One of the advantages of
the KAM results we will establish is that we do not need
to assume global non-degeneracy conditions on the map,
but rather some properties of the approximate
solution. We just need to assume that a $d\times d$ matrix is invertible.
The matrix is an explicit  algebraic expression on derivatives of
the approximate solution and averages.

\vskip.1in

Figure~\ref{fig:smcons} shows the graph of the iterates of the standard map for several values of
the perturbing parameter and for several initial conditions in each plot.

\begin{figure}[h!]
\centering
\includegraphics[width=12cm,height=14cm]{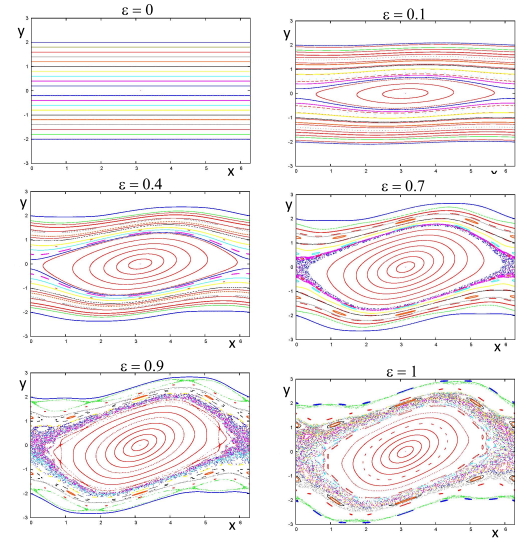}
\caption{Graphs of the conservative standard map for different values of the perturbing parameter and
different initial conditions.}
    \label{fig:smcons}
\end{figure}

From the upper left plot of Figure~\ref{fig:smcons}, we see that for $\varepsilon=0$
the system is integrable; the initial conditions has been chosen to give \sl rotational \rm quasi--periodic
curves (lying on straight lines).

When we switch-on the perturbation, even for small values as $\varepsilon=0.1$, the system becomes
non--integrable. It is easy to check that there exists a stable equilibrium point at $(\pi,0)$ and an
unstable one at $(0,0)$.
The quasi--periodic (KAM) curves are distorted with respect to the integrable case and the stable point $(\pi,0)$ is
surrounded by elliptic \sl librational \rm islands. The amplitude of the islands increases as
$\varepsilon$ gets larger, as it is shown for $\varepsilon=0.4$ where we also notice the
appearance of minor resonances. Chaotic dynamics is clearly present for $\varepsilon=0.7$
around the unstable equilibrium point, while the number of rotational quasi--periodic curves
decreases when increasing the perturbing parameter. In particular, for
$\varepsilon=0.9$ we see large chaotic regions, a few quasi--periodic curves,
new islands around higher--order periodic orbits. Finally, for $\varepsilon=1$
we have only chaotic and librational motions, while quasi--periodic curves disappear.

As we will mention in Section~\ref{sec:applications}, there is a wide literature on KAM applications to the standard map to
prove the existence of invariant rotational tori with fixed frequency, see
\cite{CellettiC95,LlaveR90,FiguerasHL17}.

\vskip.1in

The example we have presented in this Section shows a marked difference with respect to the model
that will be presented in Section~\ref{sec:DSM}, thus witnessing the divergence of the dynamical
behaviour between conservative and dissipative dynamical systems. This difference is clearly demonstrated
by the dynamics associated to the conservative and dissipative standard maps, as well as by that of more complex systems, like the
conservative and dissipative spin-orbit problems, which will be described in Section~\ref{sec:SO}.

\subsection{The dissipative standard map}\label{sec:DSM}

The dissipative standard map is obtained from \equ{SM1} adding two parameters:
a dissipative parameter $0<\lambda <1$ and a drift parameter $\mu$.
For $(y,x)\in\real\times \torus$, the equations describing
the dissipative standard map are the following:
\beqa{DSM}
y'&=&\lambda y+\mu+\varepsilon\ V(x)\nonumber\\
x'&=&x+y'\ ,
\eeqa
where $\lambda,\varepsilon\in\real_+$, $\mu\in\real$.
We remark that we obtain the conservative standard map when $\lambda =1$ and $\mu=0$.
We also remark that the Jacobian of the mapping \equ {DSM} is equal to $\lambda$, which gives a measure of the rate of
contraction or expansion of the area of the phase space.
There are several results related to the existence of attractors in the dissipative standard map;
a partial list of papers is the following:
\cite{Bohr1984,BohrB1984,Feudel,KetojaS1997,KimL1992,Schmidt19852994,VlasovaZ1984,Wenzel19916550,Yamaguchi1986307}.
Rigorous mathematical works on strange attractors for dissipative 2-D maps with twist are
\cite{Levi81,WangY08,LinY08}.

It is also important to stress that for $\varepsilon=0$ the trajectory
$\{y\equiv {\mu\over {1-\lambda }}\}\times \torus$, or equivalently
$\{\omega\equiv {\mu\over {1-\lambda }}\}\times \torus$, is invariant.
In fact, for $\varepsilon=0$ we have $y'=\lambda y+\mu$ and since we
are looking for an invariant object, we need to have $y'=y$. Hence, we
must solve the equation
\beq{y}
y=\lambda y+\mu\ .
\eeq
On the other hand, the frequency
$\omega$ associated to the standard map is, by definition, given by
$$
\omega=\lim_{j\to\infty} {x_j\over j}\ ,
$$
which gives $\omega=y$. Combining this last results with \equ{y}, we obtain
$$
\omega= {\mu\over {1-\lambda}}\ ,
$$
which shows the strong relation between the frequency and the drift, which
cannot be chosen independently.
In particular, if we fix the frequency (as it will be required in the KAM theorem of Section~\ref{sec:KAMtheorem}), then we
need to tune properly the drift parameter $\mu$. This is a substantial difference
with respect to the conservative case; dissipative dynamical systems will require a procedure to prove KAM theory
different than in the conservative case.

\vskip.1in

The dynamics associated to the dissipative standard map admits (see Figure~\ref{figsmd}) attracting
periodic orbits, invariant curve attractors; for different parameters and
initial conditions, there appear also strange
attractors which have an intricate geometrical structure (\cite{LichLieberman,WangY08}):
introducing a suitable definition of dimension, the strange
attractors are shown to have, for some parameter
values, a non--integer dimension (namely a
\sl fractal \rm dimension).
We will not consider these cases and concentrate in the cases
when the attractor is a one-dimensional smooth torus and
the motion is smoothly conjugate to a rotation.

We remark that, due to the dissipative character of
the map, there might exist at most one invariant curve attractor, while there might be
more coexisting periodic orbits (see Figure~\ref{figsmd}, panel c) and \cite{Feudel}) or strange attractors.

The  existence and breakdown of
smooth invariant tori in the dissipative standard map have been recently studied
in \cite{CCL2020} (see also \cite{CallejaF11}).

\begin{figure}[ht]
\centering
\includegraphics[width=6cm,height=4cm]{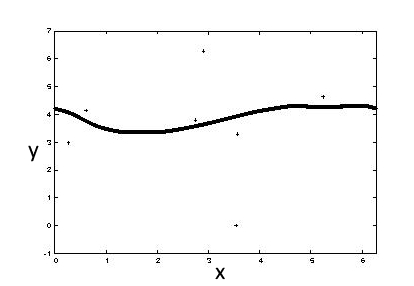}
\includegraphics[width=6cm,height=4cm]{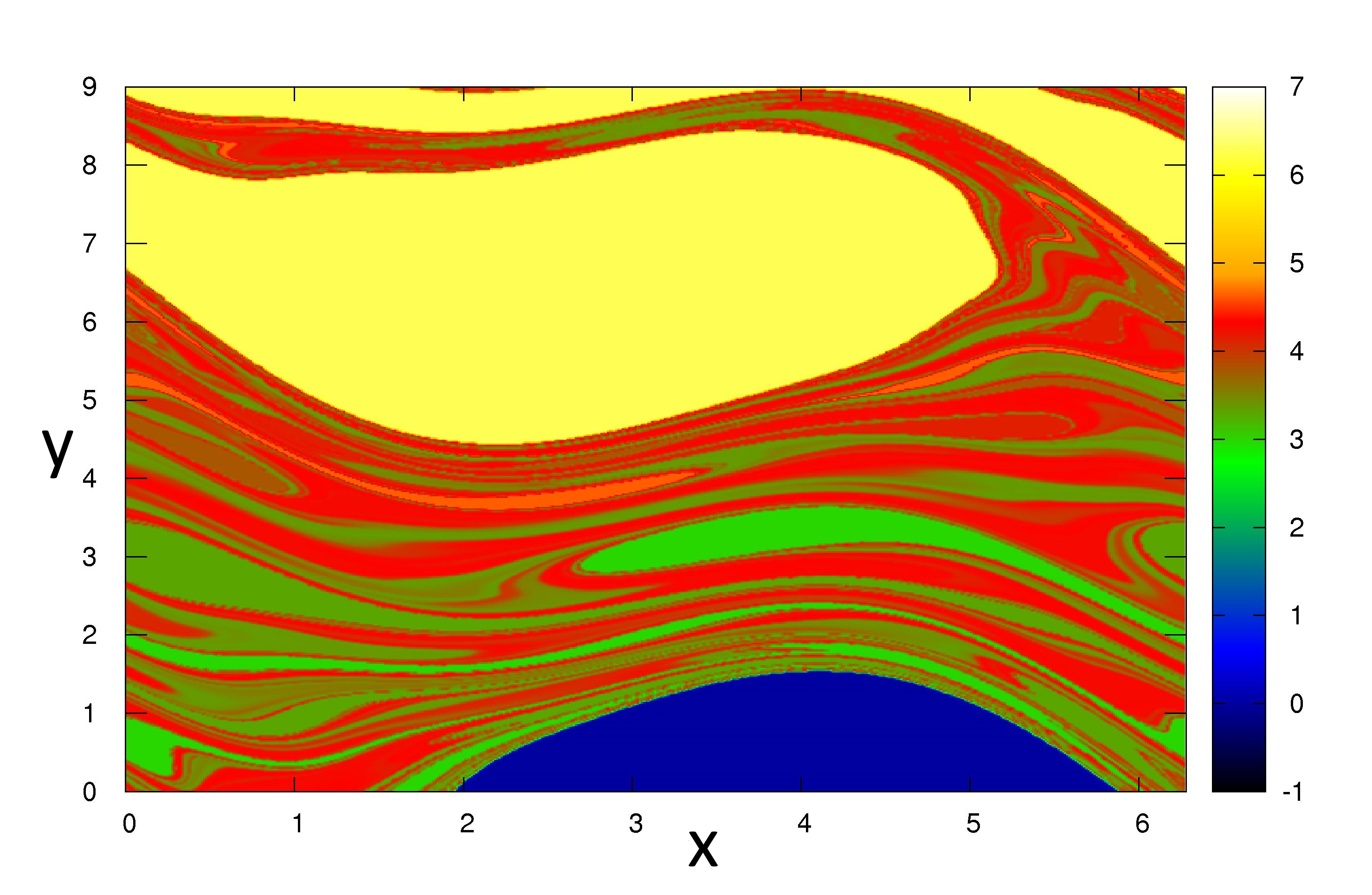}
\caption{Left: Attractors for the dissipative standard map with $\varepsilon=0.9$, $\lambda=0.91$,
$\mu=2\pi(1-\lambda)({{\sqrt{5}-1}\over 2})$.
Right: The corresponding basins of attraction using a color scale providing the frequency. }
\label{figsmd}
\end{figure}

Each of the attractors of Figure~\ref{figsmd} is characterized by an associated
\sl basin of attraction, \rm which is composed by the set of initial conditions
$(x_0,y_0)$ whose evolution ends on the given attractor. Figure~\ref{figsmd}, right, shows
the basins of attraction for the case in Figure~\ref{figsmd}, left; they have been
obtained taking $500\times 500$ random initial conditions and looking at their evolution after
having performed a number of preliminary iterations.

\vskip.1in

We want to stress that the role of the drift parameter $\mu$ is of paramount importance
in dissipative systems, since an inappropriate choice might prevent to find a
specific attractor. An example is given in Figure~\ref{figdrif}, where we look for the
torus with frequency equal to the golden ratio multiplied by the factor $2\pi$,
namely $\omega=2\pi\, {{\sqrt{5}-1}\over 2}\simeq 3.8832$, for the dissipative standard
map with $\varepsilon=0.1$, $\lambda=0.9$. The upper left panel shows that taking
$\mu=0$, the solution spirals on the point attractor at $(\pi,0)$; taking $\mu=0.1$
(Figure~\ref{figdrif}, upper right panel) leads
to an attractor which has frequency different than $\omega$, while the right choice
corresponds to $\mu=0.0617984$ as in the left bottom panel of Figure~\ref{figdrif}.
We present in Figure~\ref{figdrif}, bottom right panel, the behaviour of the drift
as a function of the dissipative parameter $\lambda$, which shows that $\mu$ tends
to zero in the limit of the conservative case, as it is expected.

\begin{figure}[h!]
\centering
\includegraphics[width=6cm,height=4.5cm]{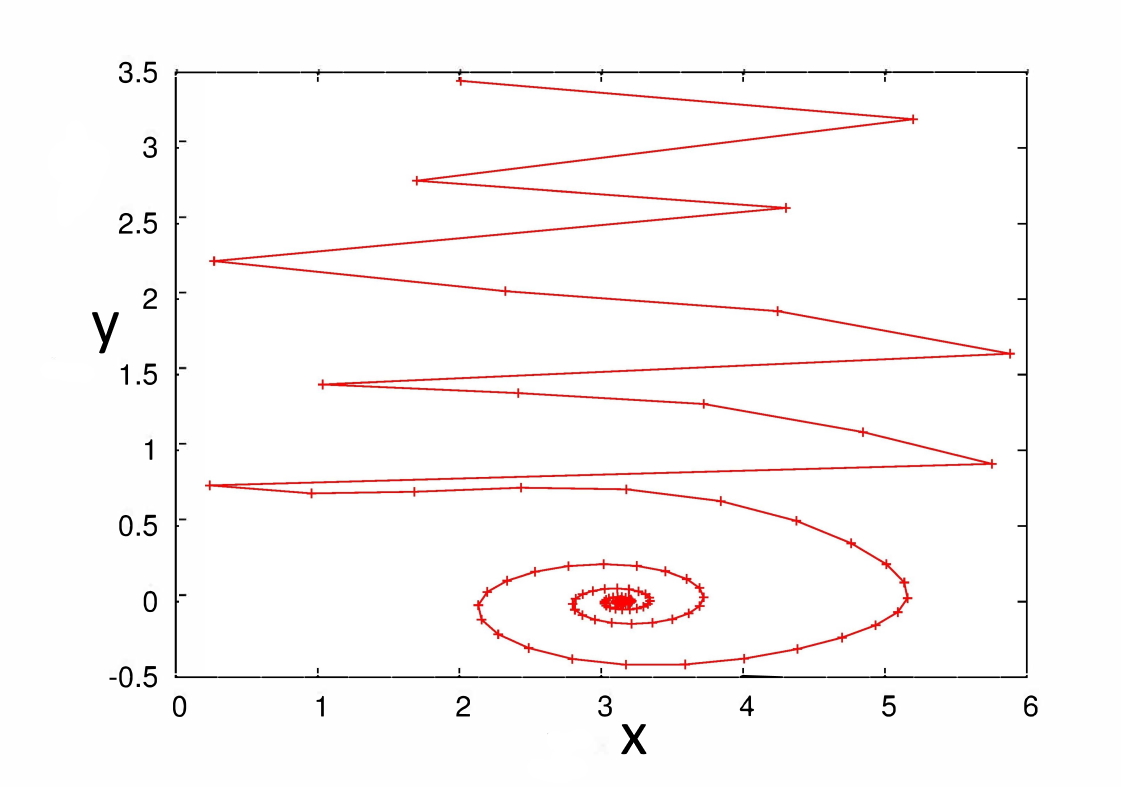}
\includegraphics[width=6cm,height=4.5cm]{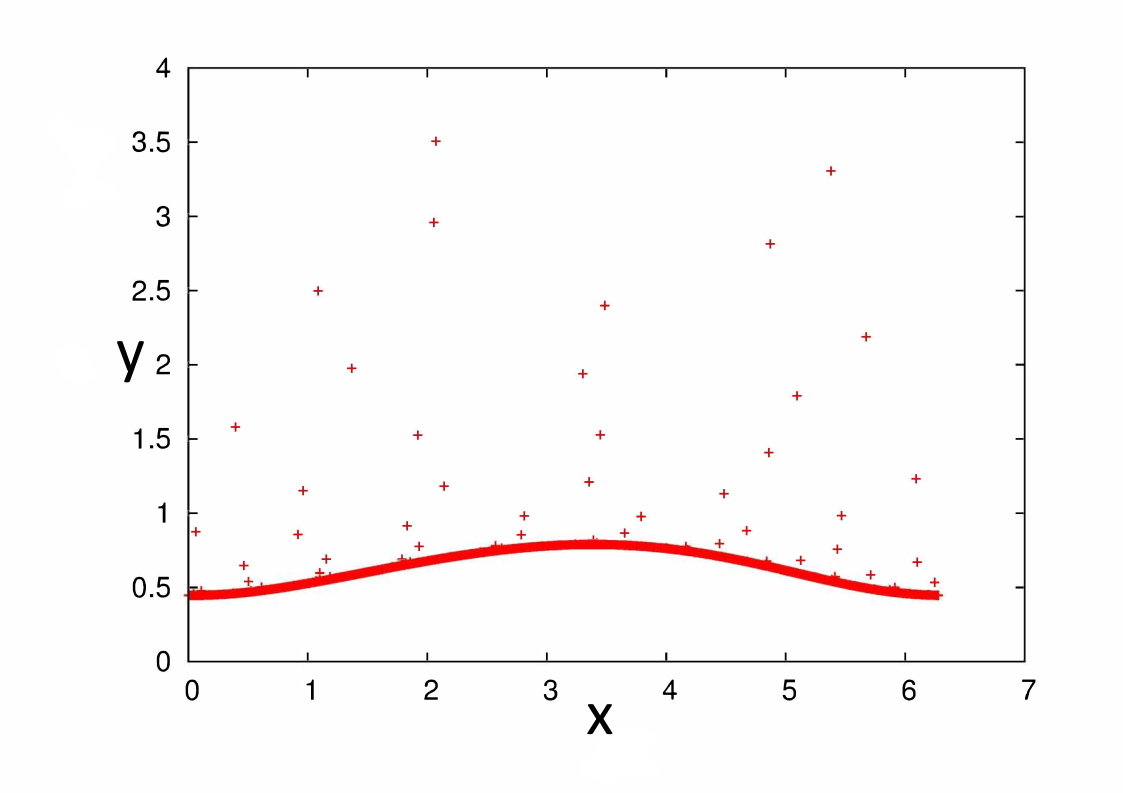}\\
\includegraphics[width=6cm,height=4.5cm]{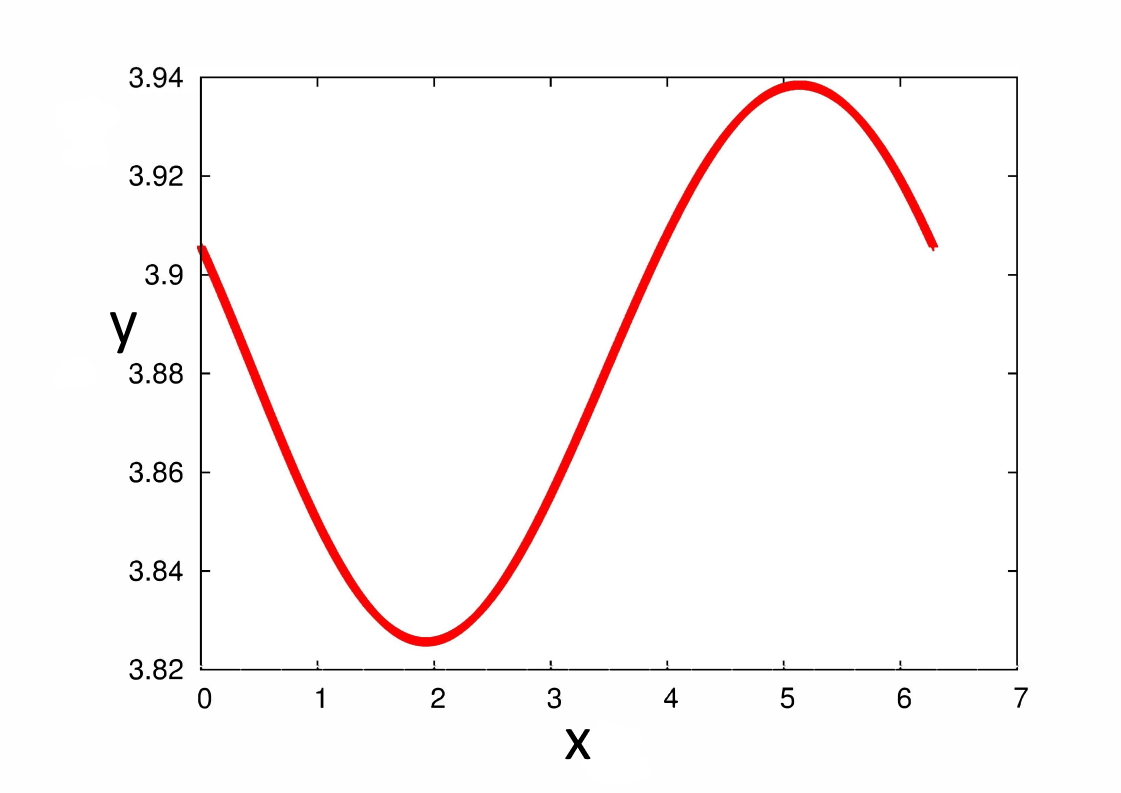}
\includegraphics[width=6cm,height=4.5cm]{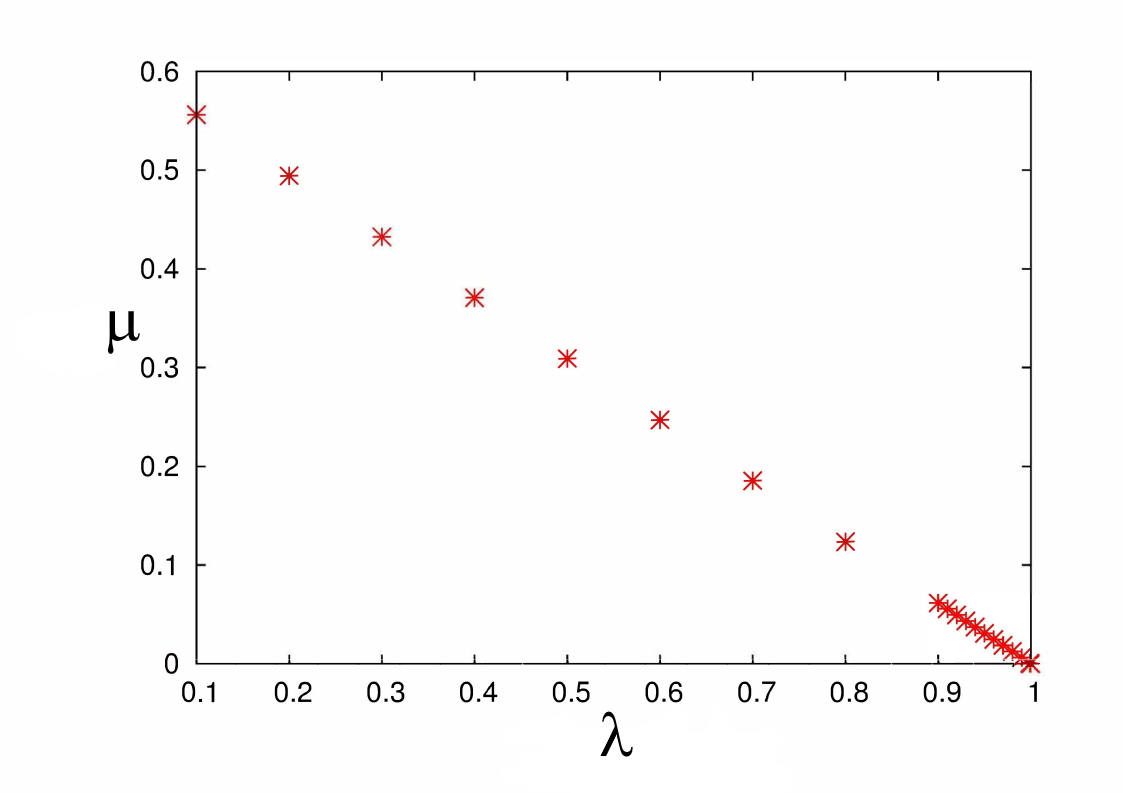}
\caption{The dissipative standard map for different values of the drift:
$\mu=0$ upper left, $\mu=0.1$ upper right, $\mu=0.0617984$ bottom left.
Graph of $\mu$ vs. $\lambda$, bottom right.}\label{figdrif}
\end{figure}

The twist condition for the dissipative standard map is a condition that now involves
the parameters. A non-twist version of the dissipative standard map is
the following map,
\beqa{DSNTM}
y'&=&\lambda y+\varepsilon\ V(x)\nonumber\\
x'&=&x+(y'-a)^2 + \mu\ .
\eeqa
In figure \ref{rotnumDSNTM}, we notice that this map has parameter values where the rotation
number does not change in a monotone direction when we change parameter $a$. See
\cite{Cal-Can-Har-20} for a study of the invariant circles of the map \equ{DSNTM}.
\begin{figure}[h!]
\centering
\includegraphics[width=6cm,height=4cm]{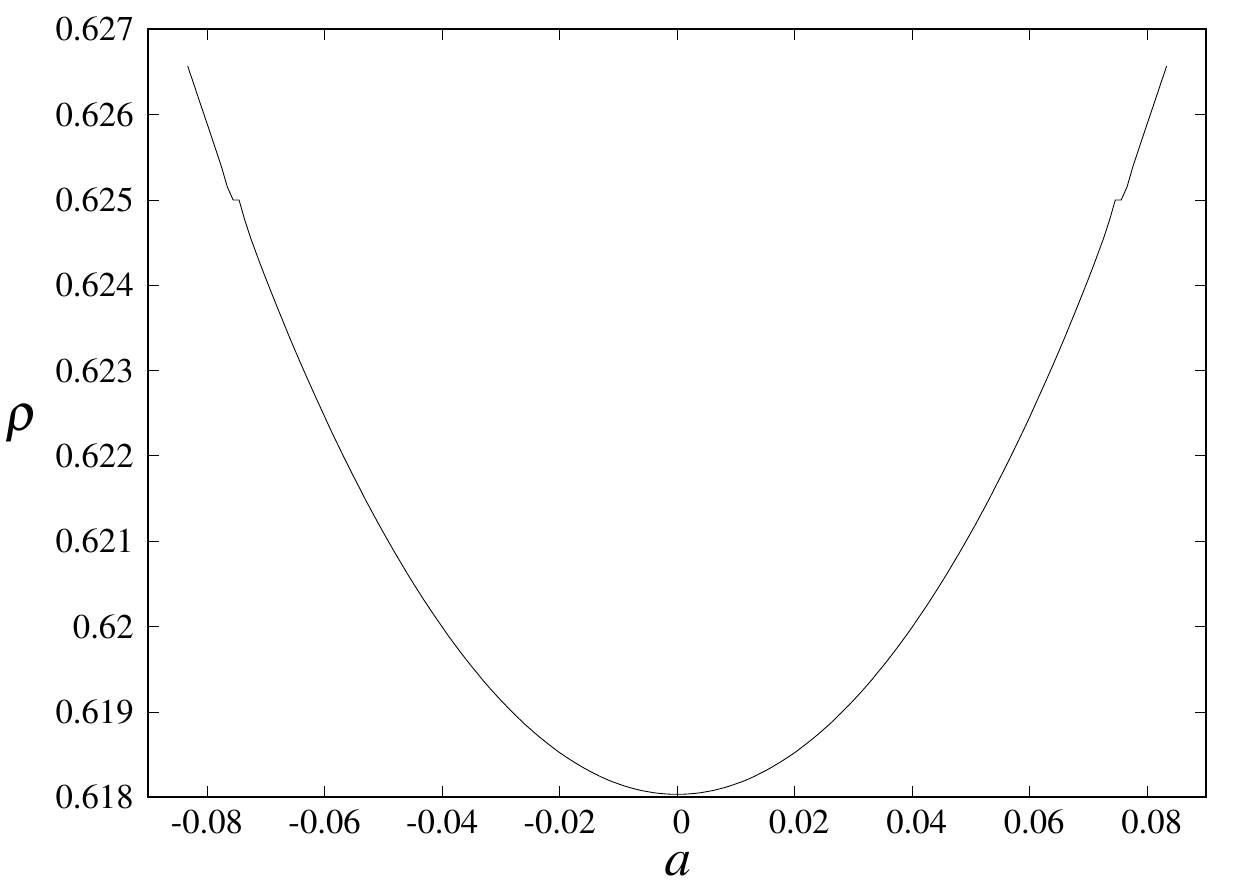}
\caption{Rotation number $\rho$ in the map \equ{DSNTM} w.r.t. the parameter $a$.
  Reproduced from \cite{Cal-Can-Har-20}.}
\label{rotnumDSNTM}
\end{figure}

\subsection{The spin-orbit problems}\label{sec:SO}
An interesting example of a continuous system which shows the main dynamical features
of regular and chaotic invariant objects is the so-called spin-orbit problem in Celestial
Mechanics. The conservative version of the model is based upon the following assumptions.
We consider a triaxial satellite, say $\cal S$, with principal moments
of inertia $I_1<I_2<I_3$. We assume that the satellite moves on a Keplerian
orbit around a central planet, say $\cal P$, while it rotates around a
spin--axis perpendicular to the orbit plane and coinciding with its
shortest physical axis.

We take a reference system centered in the planet and with the horizontal axis
coinciding with the direction of the semimajor axis.
We denote by $r$ the orbital radius, by $f$ the true anomaly, while we denote by $x$
the angle between the semimajor axis and the direction of the longest axis of the
ellipsoidal satellite (see Figure~\ref{spinorbit}).

\begin{figure}[ht]
\centering
\includegraphics[width=8cm,height=4cm]{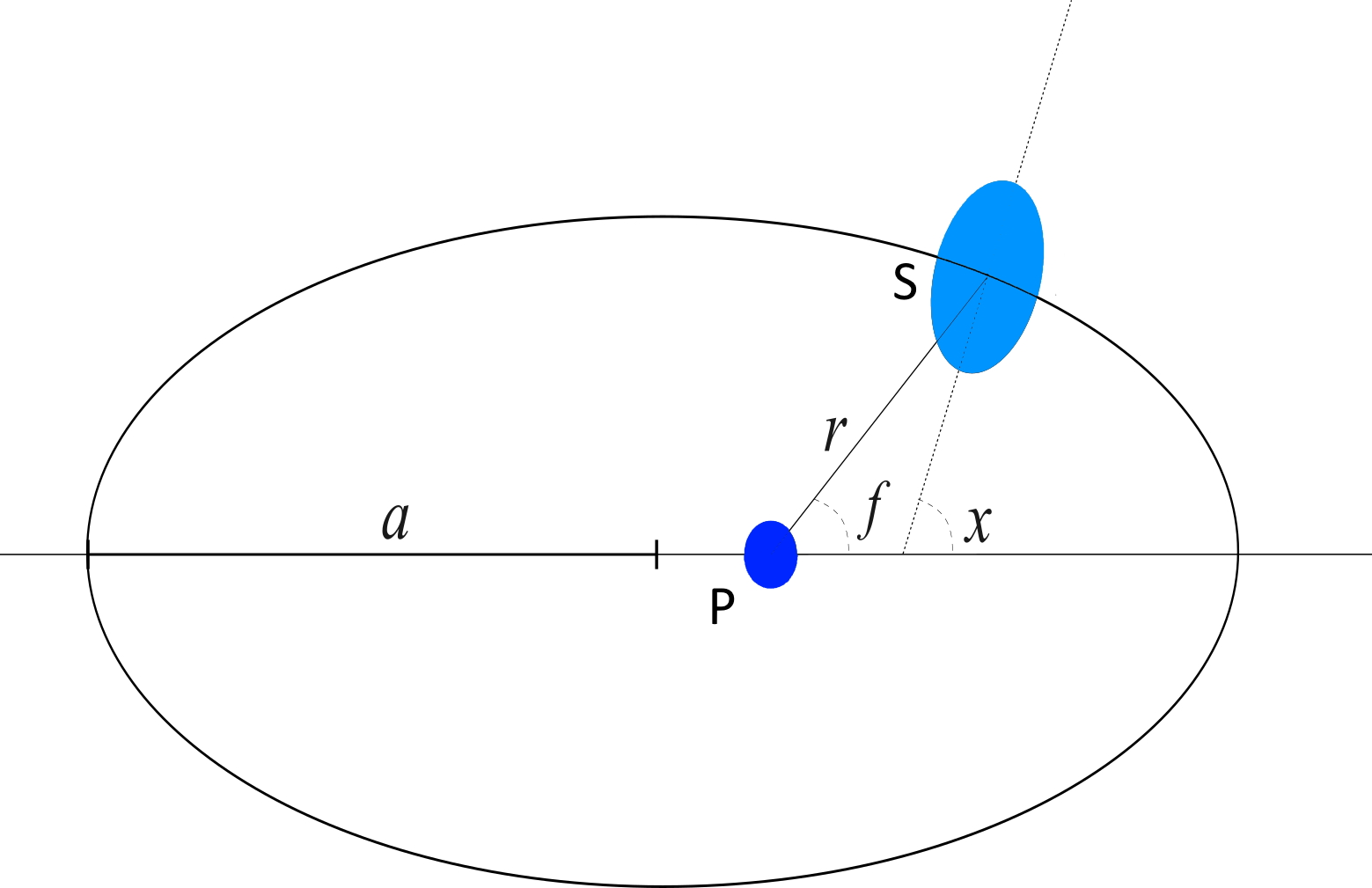}
\vglue0.1cm
\caption{The spin-orbit problem.}\label{spinorbit}
\end{figure}

The equation of motion describing the conservative spin-orbit problem is
\beq{so}
\ddot x\ +\varepsilon\ \left({a\over r}\right)^3\sin(2x-2f)=0\ ,
\eeq
where $\varepsilon={3\over 2}{{I_2-I_1}\over I_3}$ is a parameter which measures the equatorial
flattening of the satellite. Equation \equ{so} is associated to the one-dimensional, time-dependent
Hamiltonian function:
\beq{soH}
{\cal H}(y,x,t)={y^2\over 2}-{\varepsilon\over 2} \Big({a\over r(t)}\Big)^3 \cos(2x-2f(t))\ .
\eeq
Due to the assumptions of the model, the quantities $r$ and $f$ are known functions of
the time, being the solution of Kepler's problem which determines the elliptical orbit
of the satellite. They depend on the orbital eccentricity, which plays the role of
an additional parameter.

It is important to observe that:

- the Hamiltonian \equ{soH} is integrable whenever $\varepsilon=0$, namely the satellite has equatorial symmetry with $I_1=I_2$;

- the Hamiltonian \equ{soH} is integrable when the eccentricity is equal to zero, since the orbit becomes
circular, namely $r=a$ and $f$ coincides with the mean anomaly, which is proportional to time.

The existence and breakdown of invariant tori in the conservative spin-orbit problem have been investigated
in \cite{Celletti90I,Celletti90II}.

We remark that Hamilton's equations associated to \equ{soH} are given by
\beqa{SOeq}
\dot x&=&y\nonumber\\
\dot y&=&-\varepsilon \Big({a\over r(t)}\Big)^3 \sin(2x-2f(t))\ .
\eeqa
Integrating \equ{SOeq} with a modified Euler's method with time-step $h$, we obtain a
discrete system which retains the many features of the conservative standard map when taking the solution on the Poincar\'e
map at time intervals multiple of $2\pi$:
\beqano
y_{n+1}&=&y_n-\varepsilon \Big({a\over r(t_n)}\Big)^3 \sin(2x_n-2f(t_n))\ h\nonumber\\
x_{n+1}&=&x_n+y_{n+1}\ h\nonumber\\
t_{n+1}&=&t_n+h\ .
\eeqano

\vskip.1in

The dissipative spin-orbit problem is obtained by taking into account that the satellite
is non rigid and therefore it is subject to a tidal torque. The equation of motion
including a model for the tidal torque can be written as
\beq{sodiss1}
\ddot x\ +\varepsilon\ \left({a\over r}\right)^3\sin(2x-2f)=-K_d\ [L(e,t)\dot x-N(e,t)]\ ,
\eeq
where
$$
L(e,t)={a^6\over r^6}\ ,\qquad N(e,t)={a^6\over r^6}\ \dot f\
$$
(see, e.g., \cite{celletti2010,peale}).
The coefficient $K_d$ is called the dissipative constant, and depends on the physical
and orbital features of the body:
$$
K_d=3n {k_2\over {\xi Q}} \left({R_e\over a}\right)^3 {M\over m}\ ,
$$
where $n$ is the mean motion, $k_2$ is the so-called Love number (depending on the structure of the
satellite), $Q$ is called the quality factor (it compares the frequency of oscillation of the
system to the rate of dissipation of the energy), $\xi$ is a structure constant such that
$I_3 = \xi mR_e^2$ with $R_e$ the equatorial radius, $M$ is the mass of the planet, $m$ is the
mass of the satellite. For bodies like the Moon or Mercury, realistic values
are $\varepsilon=10^{-4}$ and $K_d=10^{-8}$.

The expression for the tidal torque can be simplified by assuming (as, e.g., in \cite{LaskarC}) that the
dynamics is ruled by the averages of $L(e, t)$ and $N(e, t)$ over one orbital period. The averaged
quantities are given by
\beqano
{\overline L}(e)&=&{1\over {(1-e^2)^{9\over 2}}}\ \left(1+3e^2+{3\over 8}e^4\right)\ ,\nonumber\\
{\overline N}(e)&=&{1\over {(1-e^2)^6}}\ \left(1+{{15}\over 2}e^2+{{45}\over 8}e^4+{5\over {16}}e^6\right)\ .
\eeqano
Hence, we obtain the following equation of motion in the averaged case:
\beq{SOdiss}
\ddot x+\varepsilon\ \Big({a\over r}\Big)^3
\sin(2x-2f)=-K_d\ \Big({\overline L}(e)\dot{x}-{\overline N}(e)\Big) \ ,
\eeq
We can refer to the quantity $\lambda=K_d {\overline L}(e)$ as the dissipative parameter and
to $\mu={{{\overline N}(e)}\over {{\overline L}(e)}}$ as the drift parameter.

Let us write \equ{SOdiss} in normal form as
\beqa{SOeqdiss}
\dot x&=&y\nonumber\\
\dot y&=&-\varepsilon \Big({a\over r(t)}\Big)^3 \sin(2x-2f(t))-\lambda(y-\mu)\ .
\eeqa
Similarly to the conservative case, the integration of \equ{SOeqdiss} with a modified Euler's method with time-step $h$,
leads to a discrete system similar to the dissipative standard map with dissipative and drift parameters,
when taking the solution on the Poincar\'e
map at time intervals multiple of $2\pi$:
\beqano
y_{n+1}&=&(1-\lambda h) y_n+\lambda\,\mu\,h-
\varepsilon \Big({a\over r(t_n)}\Big)^3 \sin(2x_n-2f(t_n))\ h\nonumber\\
x_{n+1}&=&x_n+y_{n+1}\ h\nonumber\\
t_{n+1}&=&t_n+h\ .
\eeqano

As we will  mention  in Section~\ref{sec:applications}, the existence and breakdown of invariant attractors
in the dissipative spin-orbit problem have been studied in \cite{CCGL2020} through an application
of KAM theory for conformally symplectic systems and through suitable numerical methods.

\section{Conformally symplectic systems and Diophantine vectors}\label{sec:CSDC}

In this section we give the definition of conformally symplectic systems for maps and flows (see
Section~\ref{sec:CS}) and we introduce the set of Diophantine vectors for discrete and continuous
systems (see Section~\ref{sec:DV}).

\subsection{Discrete and continuous conformally symplectic systems}\label{sec:CS}
An important class of dissipative dynamical systems is given by the \sl conformally symplectic \rm systems;
the dissipative standard map is an example of a conformally symplectic discrete system,
while the dissipative spin-orbit problem is an example of a conformally symplectic continuous system.

Before giving the formal definition, let us say that conformally symplectic systems are characterized by the property that they
transform the symplectic form into a multiple of itself. Beside the examples mentioned before, we stress that
conformally symplectic models can be found in different fields, e.g.
the Euler-Lagrange equations of exponentially discounted systems (\cite{Bensoussan88}, typically found in finance, when inflation is
present and one needs to minimize the cost in present money) or Gaussian thermostats (\cite{DettmannM96,WojtkowskiL98},
namely mechanical systems with forcing and a
thermostating term based on the Gauss Least Constraint Principle for nonholonomic constraints).

\vskip.1in

Let us start to introduce the notion of $2n$-dimensional conformally symplectic maps.
Let ${\cal M}  =  U\times \torus^n$ be the phase space with $U\subseteq \real^n$ an open, simply
connected domain with smooth boundary; the phase space ${\cal M}$ is endowed with
the standard scalar product and a symplectic form $\Omega$, represented by a
matrix $J$ at the point ${\underline z}$ acting on vectors ${\underline u},{\underline v}\in\real^n$
as $\Omega({\underline u},{\underline v})=({\underline u},J({\underline z}){\underline v})$ with
$(\cdot,\cdot)$ denoting the scalar product. Note that the matrix $J$
depends not only on the symplectic form but on the metric considered.

\begin{definition} \label{def:conformallysymplectic}
A diffeomorphism $f$ on $\cal M$ is \sl
conformally symplectic, if there exists a function $\lambda :
{\cal M}\to\real$ such that, denoting by $f^*$ the pull--back of $f$, we have:
\beq{def}
f^* \Omega = \lambda\Omega\ .
\eeq
\end{definition}

\vskip.1in

We remark that for $n=1$ any diffeomorphism is conformally symplectic with $\lambda$ depending on the
coordinates, namely one can take $\lambda(x)=det(Df(x))$ or $\lambda(x)=-det(Df(x))$. Instead, for
$n\ge 2$ one obtains that $\lambda$ is a constant. In fact, taking the exterior derivative of
$f^*\Omega=\lambda\Omega$, one obtains:
$$
d(f^*\Omega)=f^*\ d\Omega=0 \ =\ d\lambda\wedge \Omega+\lambda\wedge d\Omega=d\lambda\wedge \Omega\ ,
$$
which gives $d\lambda =0$; since the manifold is simply connected, then $\lambda$ is equal to a constant (see \cite{CallejaCL11}).

We also remark that for $\lambda=1$ (and $\mu=0$) we recover the symplectic case.

\vskip.1in

Let us give some explicit examples which might help to clarify the meaning of Definition~\ref{def:conformallysymplectic}.
First, we notice that we can re-formulate the notion of conformally symplectic by saying that the diffeomorphism $f$ is conformally symplectic if
\beq{CSdef}
Df^T\, J\, Df = \lambda\ J\ ,
\eeq
where the superscript $T$ denotes transposition.
In fact, from \equ{def} we have:
\beqano
f^\ast\Omega=\lambda\Omega &\Leftrightarrow&
\Omega(Df\ {\underline u},Df\ {\underline v})=\lambda\ \Omega({\underline u},{\underline v})\nonumber\\
&\Leftrightarrow& (Df\ {\underline u},J\, Df\ {\underline v})=\lambda\ ({\underline u},J\ {\underline v})\nonumber\\
&\Leftrightarrow& (u,Df^T\, J\, Df\,{\underline v})=({\underline u},\lambda\, J\ {\underline v})\nonumber\\
&\Leftrightarrow&Df^T\, J\, Df = \lambda\ J\ .
\eeqano

\vskip.1in

An example of a conformally symplectic diffeomorphism is given by the dissipative standard map.
Recalling \equ{DSM}, we have that \equ{CSdef} is satisfied, as shown below:
$$
\left(\begin{array}{cc}
  \lambda & \lambda \\
  \varepsilon V_x & 1+\varepsilon V_x \\
 \end{array}%
\right)\
\left(\begin{array}{cc}
  0 & 1 \\
  -1 & 0 \\
 \end{array}%
\right)\
\left(\begin{array}{cc}
  \lambda & \varepsilon V_x \\
  \lambda & 1+\varepsilon V_x \\
 \end{array}%
\right)=
\left(\begin{array}{cc}
  0 & \lambda \\
  -\lambda & 0 \\
 \end{array}%
\right)\
\ =\lambda\, J\ .
$$

\vskip.1in

An example of a map which does not satisfy the conformally
symplectic condition \equ{CSdef} is given by the following
4-dimensional dissipative standard map with conformal
factors $\lambda_1$, $\lambda_2$ with $\lambda_1\not=\lambda_2$:
\beqano
y_1'&=&\lambda_1 y_1+\mu_1+\varepsilon V_1(x_1,x_2)\nonumber\\
y_2'&=&\lambda_2 y_2+\mu_2+\varepsilon V_2(x_1,x_2)\nonumber\\
x_1'&=&x_1+y_1'\nonumber\\
x_2'&=&x_2+y_2'\ . \eeqano In fact, even for $\varepsilon=0$, we
obtain that \equ{CSdef} is not satisfied:
$$
Df^T\, J\, Df =
\left(\begin{array}{cccc}
0 & 0 & \lambda_1 & 0 \\
0 & 0 & 0 & \lambda_2\\
-\lambda_1 & 0 & 0 & 0\\
0 & -\lambda_2 & 0 & 0\\
\end{array}%
\right)
\not=\lambda\
\left(\begin{array}{cccc}
0 & 0 & 1 & 0 \\
0 & 0 & 0 & 1\\
-1 & 0 & 0 & 0\\
0 & -1 & 0 & 0\\
\end{array}%
\right)=
\lambda\ J \ .
$$

\vskip.1in

To conclude, we give the definition of conformally symplectic systems for continuous dynamical systems.

\begin{definition}
We say that a vector field $X$ is a conformally symplectic flow if, denoting by $L_X$ the Lie
derivative, there exists a function $\lambda:\real^{2n}\to\real$ such that
$$
L_X\Omega = \lambda \Omega\ .
$$
\end{definition}

In analogy to the definition of conformally symplectic maps, we remark that the time $t$-flow $\Phi_t$ satisfies the relation
$$
(\Phi_t)^*\Omega=e^{\lambda t}\Omega\ .
$$

\subsection{Diophantine vectors for maps and flows}\label{sec:DV}

In this Section we give the definition of Diophantine vectors for maps and flows and we
briefly recall the main properties of Diophantine vectors. We start by giving the definition
for maps.

\begin{definition}
We say that the vector $\omega\in\real^n$ satisfies the Diophantine condition, if for
a constant $C>0$ and an exponent $\tau>0$, one has
$$
\left|{{{\underline \omega}\cdot {\underline q}}\over {2\pi}}-p\right|^{-1}\
\leq\ C |q|^{\tau}\ ,\qquad p\in \integer\ ,\quad q\in \integer^n\backslash\{0\}\ .
$$
\end{definition}

\vskip.1in

In the case of flows we have the following definition.

\begin{definition}\label{def:Diophantine}
We say that the vector $\omega\in\real^n$ satisfies the Diophantine condition, if for
a Diophantine constant $C>0$ and a Diophantine exponent $\tau>0$, one has:
$$
|{\underline \omega}\cdot {\underline k}|^{-1}\ \leq\ C
|{\underline k}|^{\tau}\ ,\qquad {\underline k}\in \integer^n\backslash\{0\}\ .
$$
\end{definition}

\vskip.1in

We conclude this Section by listing below some important properties of Diophantine vectors.

$(i)$ Let us denote by $\D(C,\tau)$ the set of Diophantine vectors
satisfying Definition~\ref{def:Diophantine}.
Then, the size of the set of Diophantine vectors $\D(C,\tau)$ increases as $C$ or
$\tau$ increases.
The set of  vectors that satisfy this condition for some $C,\tau$.
is of full
Lebesgue measure in $\real^n$.

$(ii)$ There are no Diophantine vectors in $\real^n$ with $\tau<n-1$.

$(iii)$ The set of Diophantine vectors with $\tau=n-1$ in $\real^n$ has zero Lebesgue measure,
but it is everywhere dense.

$(iv)$ For $\tau>n-1$, almost every vector in $\real^n$ is
$\tau$-Diophantine, namely the complement has zero Lebesgue
measure, although it is everywhere dense.

\section{Invariant tori and KAM theory for conformally symplectic systems}\label{sec:tori}

In this Section we provide the definition of KAM (rotational) invariant tori (for maps and flows)
(see Section~\ref{sec:invtori}); the statement of the KAM theorem for conformally symplectic maps
is given in Section~\ref{sec:KAMtheorem}, whose proof is briefly recalled in Section~\ref{sec:sketch}.
The proof can be translated into a very efficient KAM algorithm (see \cite{CallejaCL11}), which is at
the basis of different results: the derivation of numerical methods to compute the breakdown threshold (Section~\ref{sec:breakdown}),
the investigation of the breakdown mechanism (Section~\ref{sec:collision}),
the implementations to specific models (see Section~\ref{sec:applications}).

\subsection{Invariant KAM tori}\label{sec:invtori}

We start by giving the definition of \sl conditionally periodic \rm and \sl quasi-periodic \rm motions.

\begin{definition}
A conditionally periodic motion is represented by a function $t\mapsto f(\omega_1 t,\ldots,\omega_n t)$, where
$f(x_1,\ldots, x_n)$ is periodic in all variables; the vector
${\underline \omega}=(\omega_1,\ldots,\omega_n)$ is called frequency.

A quasi-periodic motion is a conditionally periodic motion with incommensurable
frequencies.
\end{definition}

\vskip.1in

Next we give the following definition of \sl invariant torus. \rm

\begin{definition}
An invariant torus is an invariant manifold diffeomorphic to the standard torus $\torus^n$.
\end{definition}

We remark that any trajectory on an invariant torus carrying quasi-periodic motions is dense on the torus.
We conclude by giving the definition of (rotational) KAM torus for maps and flows.
This definition is based on the invariance equation \equ{inv} below, whose solution will be
the centerpiece of the KAM theorem presented in Section~\ref{sec:KAMtheorem}.



\vskip.1in

\begin{definition}\label{def:KAMcurve}
Let ${\cal M}\subseteq\real^n\times\torus^n$ be a symplectic manifold and let $f:{\cal M} \rightarrow {\cal M}$
be a symplectic map.
A KAM torus with frequency ${\underline \omega}\in \D(C,\tau)$ is an $n$--dimensional
invariant torus described parametrically by an embedding $K:\torus^n \rightarrow {\cal M}$, which is the
solutions of the invariance equation:
\beq{inv}
f\circ K({\underline \theta})=K({\underline \theta}+{\underline \omega})\ .
\eeq
For a family $f_\mu$ of conformally symplectic diffeomorphisms depending on a real parameter $\mu$,
a KAM attractor with frequency ${\underline \omega}\in\D(C,\tau)$ is an $n$--dimensional
invariant torus described parametrically by an embedding $K:\torus^n \rightarrow {\cal M}$ and
a drift $\mu$, which are the solutions of the invariance equation:
\beq{invCS}
f_{\mu}\circ K({\underline \theta})=K({\underline \theta}+{\underline \omega})\ .
\eeq
For conformally symplectic vector fields $X_\mu$, the invariance equation is given by
$$
X_{\mu}\circ K({\underline \theta})=({\underline \omega}\cdot \partial_{\underline \theta})\ K({\underline \theta})\ .
$$
\end{definition}

We remark that for symplectic systems the invariance equation \equ{inv} contains as only
unknown the embedding $K$, while for conformally symplectic systems the invariance equation
\equ{invCS} contains as unknowns both the embedding $K$ and the drift term $\mu$.

A graphical representation of the invariance equation \equ{inv} is given in Figure~\ref{fig:inv}.

\begin{figure}
\includegraphics[width=11cm,height=8cm]{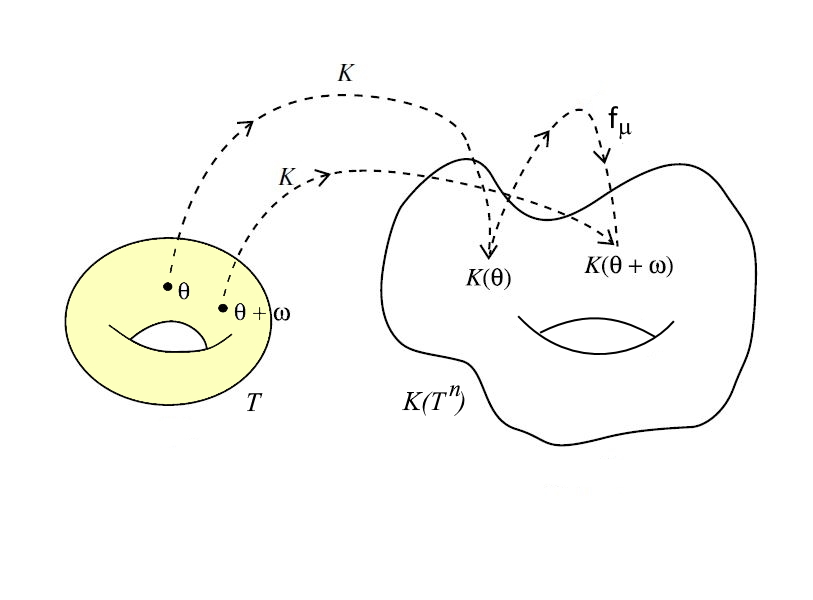}
\vglue-2cm
\caption{Geometric interpretation of the invariance equation $f_\mu\circ
K({\underline \theta})=K({\underline \theta}+{\underline \omega})$ in the unknowns $K$, $\mu$.}\label{fig:inv}
\end{figure}

Although the theory that will be presented in the next Sections apply both to maps and flows,
for simplicity of exposition we will limit to the presentation of KAM theory for maps. We refer
to \cite{CallejaCL11} for the details concerning continuous systems.

\subsection{Conformally symplectic KAM theorem}\label{sec:KAMtheorem}

We will try to answer a specific question, which is formulated below, by means of a suitable statement of
the KAM theorem; the question is motivated by many applications in several models of Celestial Mechanics, which
are often described by nearly-integrable systems. This is why we set the following question in the framework
of nearly-integrable systems, although the formulation of KAM theory does not need that the system is close to
integrable (compare with \cite{LlaveGJV05}).

\vskip.1in

Assume that a given integrable dynamical system admits an invariant torus run by a quasi-periodic motion with frequency $\underline{\omega}$
(e.g., think at Kepler's 2-body problem). Consider a perturbation of the integrable system (e.g., the restricted 3-body problem,
which is described by the 2-body problem with a perturbation proportional to the primaries' mass ratio).
The main question that we want to raise in the framework of KAM theory for nearly-integrable systems is the following:
does the perturbed system still admits an invariant torus run by a quasi-periodic motion
with the same frequency as the unperturbed system? The answer is given by celebrated KAM theory
(\cite{Kolmogorov54,Arnold63a,Moser62}), which can be implemented under very general assumptions, precisely
a non-degeneracy condition of the unperturbed system and a Diophantine condition on the frequency.

We remark that invariant tori are \sl Lagrangian: \rm if $f$ is a symplectic map and
$K$ satisfies the invariance equation \equ{inv}, then
\begin{equation}
K^* \Omega  =0\ .
\label{lagrangian}
\end{equation}
The same holds for a conformally symplectic map $f_\mu$, when $|\lambda| \ne 1$ and $K$ satisfies
the invariance equation \equ{invCS}.
If $f$ is symplectic and $\underline{\omega}$ is irrational, then the torus is Lagrangian, i.e. with maximal
dimension and isotropic (namely, the symplectic form restricted on the manifold is zero,
which implies that each tangent space is an isotropic subspace of the ambient manifold's tangent space).

\vskip.1in

Next step is to consider a nearly-integrable dynamical system affected by a dissipative force, so that the overall
system is conformally symplectic (an example is given by the spin-orbit problem with tidal torque). We assume that
the integrable symplectic system admits an invariant torus with Diophantine frequency; the question becomes
whether the non-integrable system with dissipation still admits, for suitable values of the drift parameter,
an invariant attractor run by a quasi-periodic motion with the same frequency of the unperturbed system.
The answer is given by the KAM theorem for conformally symplectic systems as given by Theorem~\ref{KAMtheorem}
(see \cite{CallejaCL11}).

\vskip.1in

Since we will be interested to give explicit estimates in specific model problems, we introduce the following norms
for analytic and differentiable functions.

\begin{definition}
\sl Analytic norm. \rm Given $\rho >0$, we define the complex extension of the torus, say $\torus_\rho^n$, as the set
$$
\torus_\rho^n = \{ {\underline \theta}\in \complex^n/(2\pi\integer)^n:\ {\rm Re}({\underline \theta})\in\torus^n,\
|{\rm Im}(\theta_j)| \le \rho\ ,\ \ j=1,...,n \}\ ;
$$
we denote by ${\cal A}_{\rho}$ the set of analytic functions in $Int(\torus_\rho^n)$ with the norm
$$
\|f\|_{\rho} = \sup_{{\underline \theta}\in\torus_\rho^n}\ |f({\underline \theta})|\ .
$$
\sl Sobolev norm. \rm For a function $f=f({\underline z})$ expanded in Fourier series as $f({\underline z}) =
\sum_{{\underline k}\in{\integer}^n} \widehat{f_{\underline k}}\ e^{2\pi i{\underline k}\cdot {\underline z}}$
for an integer $m>0$, we define the space $H^m$ as
$$
H^m = \Big\{ f:\torus^n\to \complex \ :\ \| f\|_{m}^2 \equiv
\sum_{{\underline k}\in{\integer}^n} |\,\widehat{f_{\underline k}}\,|^2 (1+|{\underline k}|^2)^m <
\infty\Big\}\ . \label{sobolevnorm}
$$
\end{definition}

Borrowing the statement from \cite{CallejaCL11}, we give below the formulation of the KAM theorem for
conformally symplectic systems (see \cite{LlaveGJV05} for the statement for symplectic systems).
We give the statement for maps, although the results can be formulated also for
systems with continuous time (flows).
Indeed in \cite{CallejaCL11} one can find
a construction that shows that the
results for maps imply the results for flows as well as  direct
proof of the results for flows.

\begin{theorem}\label{KAMtheorem}
Let ${\underline \omega}\in \D(C,\tau)$, $f_\mu:\real^n\times\torus^n\rightarrow \real^n\times\torus^n$ be a
conformally symplectic diffeomorphism, and let $(K,\mu)$ be an approximate solution of the invariance
equation \equ{invCS} with error term $E$:
$$
f_{\mu}\circ K({\underline \theta})-K({\underline \theta}+{\underline \omega})=E({\underline \theta})\ .
$$
Let $N$ be the quantity
\beq{N}
N({\underline \theta}) = (DK({\underline \theta})^T DK({\underline \theta}))^{-1}
\eeq
and let $M({\underline \theta})$ be the $2n\times 2n$ matrix defined by
$$
M({\underline \theta}) = [ DK({\underline \theta})\ |\  J(K({\underline \theta}))^{-1}\ DK({\underline \theta}) N({\underline \theta})]\ .
$$
Let $P({\underline \theta})$ be defined as
$$
P(\theta)\equiv DK(\theta)N(\theta)\ ;
$$
let $A(\theta) \equiv \lambda\,\Id$ and let $S({\underline \theta})$ be
\beq{S}
S(\theta) \equiv P(\theta+\omega)^T Df_\mu \circ K(\theta) J^{-1}\circ K(\theta)P(\theta)
- N(\theta+\omega)^T \gamma(\theta + \omega) N(\theta+\omega) A(\theta)
\eeq
with
$$
\gamma(\theta) \equiv DK(\theta)^T J^{-1} \circ K(\theta)
DK(\theta)\ .
$$
Assume that the following non--degeneracy condition is satisfied:
\beq{ND}
\det
\left(
\begin{array}{cc}
  \langle S\rangle & {\langle {SB^0}}\rangle +\langle {\widetilde A_1}\rangle  \\
  (\lambda-1){\rm Id} & \langle {\widetilde A_2}\rangle  \\
 \end{array}%
\right) \ne 0
\eeq
with $\widetilde A_1$, $\widetilde A_2$ the first and second $n$ columns
of $\tilde A=M^{-1}({\underline \theta}+{\underline \omega}) D_{\mu} f_{\mu} \circ K$,
$B^0=B-\langle B\rangle$ is the solution of $\lambda B^0({\underline \theta})-B^0({\underline \theta}+{\underline \omega})=-(\widetilde A_2)^0({\underline \theta})$.

For $\rho>0$, let $0<\delta<{\rho\over 2}$; if the solution is sufficiently approximate, namely
$$
\|E\|_\rho\leq C_3\ C^{-4}\ \delta^{4\tau}
$$
for a suitable constant $C_3>0$, then there exists an exact solution $(K_e,\mu_e)$, such that
$$
\|K_e-K\|_{\rho-2\delta}\leq C_4\ C^2\ \delta^{-2\tau}\ \|E\|_{\rho}\ ,\quad
|\mu_e-\mu|\leq C_5\ \|E\|_{\rho}
$$
with suitable constants $C_4,C_5>0$.
\end{theorem}

\begin{remark}
It is useful to make a remark on the non--degeneracy condition \equ{ND}, when applied to the conservative
and dissipative standard maps (\equ{SM1}, \equ{DSM}).
For the conservative standard map, the non--degeneracy condition is typically the so-called \sl twist \rm condition,
which can be written as
\beq{ND1}
{{\partial x'}\over {\partial y}}\not=0\ ,
\eeq
implying that the lift of the map transforms any vertical line always on the same side.

Instead, for the dissipative standard map, that we modify adding a generic dependence on the drift
through a function $p=p(\mu)$, say
\beqano
y'&=&\lambda y+p(\mu)+\varepsilon\  V(x)\nonumber\\
x'&=&x+y'\ ,
\eeqano
then the non--degeneracy condition involves the twist condition and
a non--degeneracy condition with respect to to the parameters, namely:
\beq{ND2}
{{\partial x'}\over {\partial y}}\not=0\ ,\qquad\qquad {{dp(\mu)}\over {d\mu}}\not=0\ .
\eeq
We remark, however,
that \equ{ND1} and \equ{ND2} involve global properties of the system,
while \equ{ND} is a condition involving just the approximate solution,
so that \equ{ND} may be applied in situations where \equ{ND}, \equ{ND2}
fail.
\end{remark}

The proof of Theorem~\ref{KAMtheorem} is given in \cite{CallejaCL11} through the
a-posteriori approach developed in \cite{LlaveGJV05} and making use of an adjustment of parameters
(see \cite{Moser67,BroerHS96b}): assume we can find an approximate
solution $(K,\mu)$ of the invariance equation, satisfying a
non-degeneracy condition, then we can find a true solution $(K_e,\mu_e)$ close to $(K,\mu)$, such that
$\|K_e-K\|$, $|\mu_e-\mu|$ is small. A sketch of the proof of Theorem~\ref{KAMtheorem}
is presented in Section~\ref{sec:sketch}.

\vskip.1in

We conclude this Section by remarking that the a-posteriori approach presents several advantages,
among which:

\begin{itemize}
  \item
    $(i)$ it can be developed in any coordinate frame and not necessarily in action-angle variables. In many practical problems, the action-angle
    variables are difficult to compute and involve complex singularities.

    Of course, once we have the existence of the torus, we can
    construct action angle variables. Hence, compared to more standard
    results, the accomplishment of the method is that the existence
    of action variables and the quasi-integrablility is
    moved from the hypothesis to the conclusions. This is
    useful in practice  since the hypothesis is  hard to verify in applications.

    Of course, moving conclusions to hypothesis without regards on
    how to check them, one can get the same conclusions. In that
    respect, adding the hypothesis of the existence of the torus
    would get a theorem with the same applicability and a simpler proof.

        \item
      $(ii)$ The system is not assumed to be nearly integrable.
      \item
$(iii)$ Instead of constructing a sequence of coordinate transformations
on shrinking domains as in the perturbation approach, one computes suitable corrections
to the embedding and the drift.

The computation of the embeddings requires
to work only with variables of $n$ dimensions whereas transformation theory
requires to work with variables in $2n$ dimensions. The complexity
of representing functions grows exponentially -- with a large exponent --
on the dimension. The composition
of two functions has rather awkward analytic and numerical properties.

\item
  $(iv)$ The non-degeneracy assumptions
are not global properties of the map, but are rather properties
of the approximate solution considered.
\item
  $(v)$ One does
not need to justify how the approximate solution was obtained.
In particular, one can take as approximate solution the
result of numerical calculations or a formal expansion.

Verifying the hypothesis in a numerical approximation is
just a finite number of calculations.  Even if this number is
too large to do by hand, it could be moderate to do with a computer
(e.g., a few hours in a common laptop).  If these can be
done taking care of roundoff and truncation errors, this may lead
to a computer assisted proof.

One can also verify the hypothesis easily in a numerical expansion.

Note that, in these verifications it is crucial
as noted in $(i)$ to move the difficult to verify facts from
the hypothesis that need to be checked to the conclusions.
\end{itemize}

\subsection{A sketch of the proof of the KAM theorem}\label{sec:sketch}

The proof of Theorem~\ref{KAMtheorem} can be summarized as composed by five main steps:

\begin{itemize}
\item Step 1: starting from an approximate solution, write the linearization of the invariance equation.

\item Step 2: by a Newton's method find a quadratically smaller approximation.

\item Step 3: under a non--degeneracy condition, solve the cohomological equation that allows to find the new approximation.

\item Step 4: iterate the procedure and show its convergence.

\item Step 5: prove that the solution is locally unique.
\end{itemize}

\vskip.1in

We briefly describe such steps as follows.

\subsubsection{Step 1: approximate solution and linearization.}

Let $(K,\mu)$ be an approximate solution satisfying
\beq{invE}
f_\mu\circ K({\underline \theta})-K({\underline \theta}+{\underline \omega})=E({\underline \theta})\ .
\eeq
Using the Lagrangian property $K^*\Omega=0$ written in coordinates, namely
$$
DK^T({\underline \theta})\ J\circ K({\underline \theta})\ DK({\underline \theta})=0\ ,
$$
we get that the tangent space is given by
\beq{tangent}
Range \Big(DK({\underline \theta})\Big) \oplus Range \Big(V({\underline \theta})\Big)
\eeq
with $N$ as in \equ{N} and
$$
V({\underline \theta})=J^{-1}\circ K({\underline \theta})\ DK({\underline \theta}) N({\underline \theta})\ .
$$
Define the quantity
\beq{Mdef}
M({\underline \theta})=[ DK({\underline \theta})\ |\  V({\underline \theta})]\ .
\eeq
Then, we have the following result.

\begin{lemma}\label{lem:R}
Up to a remainder $R$, we have the following relation:
$$
Df_\mu \circ K({\underline \theta})\ M({\underline \theta}) = M({\underline \theta} +{\underline \omega})
\left(%
\begin{array}{cc}
  \Id & S({\underline \theta})\\
  0 & \lambda \Id \\
\end{array}%
\right)
+R({\underline \theta})\ .
$$
\end{lemma}

\begin{proof}
Recalling the definition of $M$ in \equ{Mdef}, we have that taking the derivative of
$$
f_\mu\circ K({\underline \theta})=K({\underline \theta}+{\underline \omega})+E({\underline \theta})\ ,
$$
one obtains the relation
$$
Df_\mu \circ K({\underline \theta})\ DK({\underline \theta})= DK({\underline \theta}+{\underline \omega})+DE({\underline \theta})\ .
$$
Due to \equ{tangent}, one obtains:
$$
Df_\mu\circ K({\underline \theta})\ V({\underline \theta})=D K({\underline \theta}+{\underline \omega})\
S({\underline \theta})+V({\underline \theta}+{\underline \omega})\ \lambda\ {\Id}+ h.o.t.
$$
with $S$ as in \equ{S}.
\end{proof}

\subsubsection{Step 2: determine a new approximation.}

Let the new approximation $(K',\mu')$ be defined as $K'=K+M W$, $\mu'=\mu+\sigma$. Let $E'$ be the
error associated to $(K',\mu')$:
\beq{APPR-INVp}
f_{\mu'}\circ K'({\underline \theta})-K'({\underline \theta}+{\underline \omega}) =E'({\underline \theta})\ .
\eeq
Expanding \equ{APPR-INVp} in Taylor series, we get
\beqano
&&f_\mu\circ K({\underline \theta})+Df_\mu\circ K({\underline \theta})\ M({\underline \theta})W({\underline \theta})+D_\mu  f_\mu \circ K({\underline \theta})\sigma\nonumber\\
&&\qquad -K({\underline \theta}+{\underline \omega})- M({\underline \theta}+{\underline \omega})\
W({\underline \theta}+{\underline \omega})+h.o.t.=E'({\underline \theta})\ .
\eeqano
Recalling \equ{invE}, the new error $E'$ is quadratically smaller provided the following relation holds:
\beq{quad}
Df_\mu \circ K({\underline \theta})\ M({\underline \theta})W({\underline \theta}) - M({\underline \theta}+
{\underline \omega})\ W({\underline \theta}+{\underline \omega}) + D_\mu  f_\mu \circ K({\underline \theta})\sigma =
- E({\underline \theta})\ .
\eeq
Combining \equ{quad} and Lemma~\ref{lem:R}, we have:
$$
Df_\mu \circ K({\underline \theta})\ M({\underline \theta}) = M({\underline \theta} +{\underline \omega})
\left(%
\begin{array}{cc}
  \Id & S({\underline \theta}) \\
  0 & \lambda \Id \\
\end{array}%
\right)+R({\underline \theta})\ .
$$
This allows to get the following equations for $W=(W_1,W_2)$ and $\sigma$
$$
M({\underline \theta} +{\underline \omega})
\left(%
\begin{array}{cc}
  \Id & S({\underline \theta} \\
  0 & \lambda \Id \\
\end{array}%
\right)
\ W({\underline \theta}) - M({\underline \theta}+{\underline \omega})\ W({\underline \theta}+{\underline \omega}) =
- E({\underline \theta})- D_\mu  f_\mu \circ K({\underline \theta})\sigma
$$
that we are going to  make more explicit.
Multiplying by $M({\underline \theta}+{\underline \omega})^{-1}$ and writing $W=(W_1,W_2)$, one gets that the previous equation is equivalent to:
\beq{W12}
\left(%
\begin{array}{cc}
  \Id & S({\underline \theta}) \\
  0 & \lambda \Id \\
\end{array}%
\right)
\left(\begin{array}{c}
  W_1({\underline \theta}) \\
  W_2({\underline \theta}) \\
\end{array}\right)-
\left(\begin{array}{c}
  W_1({\underline \theta}+{\underline \omega}) \\
  W_2({\underline \theta}+{\underline \omega}) \\
\end{array}\right)=
\left(\begin{array}{c}
  -\tilde E_1({\underline \theta})-\tilde A_1({\underline \theta})\sigma \\
  -\tilde E_2({\underline \theta})-\tilde A_2({\underline \theta})\sigma \\
\end{array}\right)
\eeq
with $\tilde E_j({\underline \theta})=-(M({\underline \theta} +{\underline \omega})^{-1}E)_j$,
$\tilde A_j({\underline \theta})=(M({\underline \theta} +{\underline \omega})^{-1}D_\mu f_\mu\circ K)_j$.
Writing \equ{W12} in components, we obtain:
\beqa{B}
W_1({\underline \theta})-W_1({\underline \theta}+{\underline \omega}) &=& -\widetilde
E_1({\underline \theta})-S({\underline \theta}) W_2({\underline \theta})-\widetilde A_1({\underline \theta})\,
\sigma\nonumber\\
\lambda W_2({\underline \theta})-W_2({\underline \theta}+{\underline \omega}) &=& -\widetilde
E_2({\underline \theta})-\widetilde
A_2({\underline \theta})\,\sigma\ .
\eeqa
The cohomological equations \equ{B} allow to find the corrections $W_1$, $W_2$ and $\sigma$, as sketched in the next step.

\subsubsection{Step 3: solve the cohomological equations.}

To determine the new approximation, we need to solve equations \equ{B},
which are equations with constant coefficients for
$W_1$, $W_2$ and $\sigma$ for known $S$, $\widetilde E\equiv(\widetilde E_1,\widetilde E_2)$, $\widetilde
A\equiv[\widetilde A_1|\ \widetilde A_2]$.

The first equation  in \equ{B} is a standard small divisor equation, which can be solved under the Diophantine condition
on the frequency, so to bound the small divisors.

For $|\lambda|\ne1$ and for all real vectors ${\underline \omega}$, it is possible to solve the
second equation in \equ{B} by an elementary contraction mapping argument.

We remark that, using Cauchy estimates for the cohomological equations \equ{B}, we can bound
$\|W_1\|_{\rho-\delta}$ and $\|W_2\|_{\rho-\delta}$ by $\|E\|_\rho$.

To solve the cohomological equations, we proceed as follows.
Take the averages of each equation in \equ{B} and use the
non--degeneracy condition to determine $\langle W_2\rangle$, $\sigma$ by solving the equation
$$
\left(
\begin{array}{cc}
  \langle S\rangle & {\langle {SB^0}}\rangle +\langle {\widetilde A_1}\rangle  \\
  (\lambda-1){\rm Id} & \langle {\widetilde A_2}\rangle  \\
 \end{array}%
\right)
\left(
\begin{array}{cc}
\langle W_2\rangle  \\
\sigma\\
 \end{array}%
\right)=
\left(
\begin{array}{cc}
-\langle S \tilde B^0 \rangle- \langle \widetilde E_1\rangle \\
-\langle \widetilde {E_2}\rangle \\
 \end{array}%
\right)\ ,
$$
where we have split $W_2$ as $W_2=\langle W_2\rangle+B^0+\sigma \tilde B^0$.

\vskip.1in

Next, we need to solve the second equation in \equ{B} for $W_2$, which is an equation of the form
$\lambda W_2({\underline \theta})-W_2({\underline \theta}+{\underline \omega})=Q_2({\underline \theta})$
with $Q_2$ known. Such equation is always solvable for any $|\lambda|\not=1$ by a contraction mapping argument, using
that
$\lambda W_2({\underline \theta})-W_2({\underline \theta}+{\underline \omega}) =\sum_{\underline k} \widehat W_{2,{\underline k}}\,
e^{i{\underline k}\cdot{\underline \theta}}(\lambda-e^{i{\underline k}\cdot {\underline \omega}})$.

\vskip.1in

Finally, we solve the first equation  in \equ{B} for $W_1$, which amounts to solve an equation
of the form $W_1({\underline \theta})-W_1({\underline \theta}+{\underline \omega})=Q_1({\underline \theta})$
with $Q_1$ known. It involves small (zero) divisors, since for ${\underline k}={\underline 0}$ one has
$1-e^{i{\underline k}\cdot {\underline \omega}}=0$. The left hand side of the first equation in \equ{B}
can be expanded as
$$
W_1({\underline \theta})-W_1({\underline \theta}+{\underline \omega}) =\sum_{{\underline k}\in\integer^n
\backslash\{{\underline 0}\}} \widehat W_{1,{\underline k}}\,
e^{i{\underline k}\cdot{\underline \theta}}(1-e^{i{\underline k}\cdot {\underline \omega}})\ .
$$
To get a bound for the solution of \equ{B}, we need the following result.

\begin{proposition}\label{pro:cauchy}
Let $Z=Z({\underline \theta})$ be a function
with zero average and such that $Z\in{\cal A}_\rho$ or $Z\in H^m$.
Let ${\underline \omega}\in D(C,\tau)$. Assume that the function $U=U({\underline \theta})$ satisfies
$$
\lambda U({\underline \theta})-U({\underline \theta}+{\underline \omega})=Z({\underline \theta})\ .
$$
Then, if $\lambda\not=1$, $|\lambda|\in[A,A^{-1}]$ for $0<A<1$, we have that
$$
\|U({\underline \theta})\|_{\rho-\delta}\leq C\delta^{-\tau} \|Z\|_\rho\ .
$$
\end{proposition}

\vskip.1in

We refer to \cite{CallejaCL11,Russmann76a} for the proof of Proposition~\ref{pro:cauchy}.

\subsubsection{Step 4: convergence of the iterative step.}

The solution described in Step 3, allows to state that the invariance equation is satisfied with an error quadratically smaller, i.e.
$$
\|E'\|_{\rho-\delta}\leq C_8\delta^{-2\tau}\|E\|_\rho^2\ ,\qquad
\|E'\|_{H^{m-\tau}}\leq C_9\|E\|_{H^m}^2\ .
$$
The procedure at Step 3 can be iterated to get a sequence of approximate solutions,
say $\{K_j,\mu_j\}$. Its convergence is obtained through an abstract implicit function theorem,
alternating the iteration with carefully chosen smoothings operators defined in
a scale of Banach spaces (analytic functions or Sobolev spaces).

\subsubsection{Step 5: local uniqueness.}

Under smallness conditions, one can prove that, if there exist two solutions $(K_a,\mu_a)$, $(K_b,\mu_b)$, then there exists
${\underline \psi}\in\real^n$ such that
$$
K_b({\underline \theta})=K_a({\underline \theta}+{\underline \psi})\qquad {\rm and}\qquad \mu_a=\mu_b\ .
$$
We remark that in the analytic case, the smoothing is obtained by rescaling the size of the strip on which
the analytic functions are defined at each step, given that the domains where they are defined shrink by a given amount.
Then, for the sequence of solutions $\{K_j,\mu_j\}$, one can take the analyticity domain parameters $\rho_h$ and the
shrinking parameters $\delta_h$ as
$$
\rho_0=\rho\ ,\qquad \delta_h={\rho_0\over {2^{h+2}}}\ ,\qquad \rho_{h+1}=\rho_h-\delta_h\ ,\qquad h\geq 0\ .
$$
Given that the error is quadratic, we can write for some $a,b>0$ and a constant $C_E>0$:
$$
\|E(K_{h+1},\mu_{h+1})\|_{\rho_{h+1}}\leq C_E\ \nu^a\delta_h^b\ \|E(K_h,\mu_h)\|_{\rho_h}^2\ .
$$
If the quantity $\varepsilon_0\equiv \|E(K_0,\mu_0)\|_{\rho_0}$ is small enough, then one can prove that
$$
\|K_h-K_0\|_{\rho_h}\leq C_K\varepsilon_0\ , \qquad |\mu_h-\mu_0|\leq C_\mu\varepsilon_0
$$
for some constants $C_K,C_\mu>0$.
A finite number of conditions on parameters and norms will imply the indefinite iterability
of the procedure and its convergence.

\vskip.1in

The a-posteriori approach for conformally symplectic systems has a number of consequences and further developments that
we briefly summarize below, referring to the cited literature for full details:

\begin{itemize}
    \item the method provides an efficient algorithm to determine the breakdown threshold,
very suitable for computer implementations (\cite{CallejaC10});

    \item the a-posteriori method allows to find rigorous  lower estimates of the
      breakdown threshold (\cite{Rana87,FiguerasHL17}). The
      rigorous lower estimates for symplectic
      maps in \cite{Rana87,FiguerasHL17} are very close to
      to  the rigorous upper estimates in \cite{Jungreis91}.
      In \cite{CCL2020} one can find very detailed estimates (they do not
      control completely the round off error, but they control everything else),
      that are comparable with the best numerical estimates computed by other
      methods;

    \item one gets that the local behavior near quasi--periodic solutions is given by a rotation in
the angles and a shrink in the actions (\cite{CallejaCL11b});

 \item the method allows to obtain a partial justification of Greene's criterion
for the computation of the breakdown threshold of invariant attractors (\cite{CCFL14});

 \item one obtains a bootstrap of regularity, which allows to state that all smooth enough tori
are analytic, whenever the map is analytic (\cite{CallejaCL11});

 \item one gets a characterization of the analyticity domains of the quasi--periodic attractors
in the symplectic limit (\cite{CCLdomain});

 \item one can prove the existence of whiskered tori for conformally symplectic systems (\cite{CCLwhiskered}).
\end{itemize}

\vskip.1in

Concerning the first item above, we stress that
the proof given in \cite{CallejaCL11} leads to a very efficient KAM algorithm, which can be
implemented numerically and it is shown to work very close to the boundary of validity (\cite{CCL2020}).
Indeed, all steps of the algorithm involve diagonal operations in the Fourier space and/or diagonal
operations in the real space.
Moreover, if we represent a function in discrete points or
in Fourier space, then we can compute the other functions by applying
the Fast Fourier Transform (FFT). Using $N$ Fourier modes
to discretize the function, then we need $O(N)$ storage and $O(N\log N)$ operations.   Note that all the steps in the algorithm can be implemented in
a few lines in a high level language so the the resulting algorithm
is not very hard to implement ( about 200 lines in {\tt Octave} and
about 2000 lines in C).  Even if the above transcription of the
algorithm works extremely well in near integrable systems, when
approaching the breakdown, one needs to take some standard precautions
(e.g. monitoring the size of the tails of Fourier series).

We also remark that the KAM proof requires a computer to make very long computations, which are
needed to determine, for example, the initial approximate solution or to check the KAM algorithm.
However, the computer introduces rounding-off and propagation errors, which can be controlled
through interval arithmetic for which we refer to the specialized literature
(see, e.g., \cite{Meyer91,MR89d:58095,KochSW96,Haro}).

\section{Breakdown of quasi--periodic tori and quasi--periodic attractors}\label{sec:breakdown}

The analytical estimates which can be obtained through the implementation of the KAM theorem
represent a rigorous lower bound of the breakdown threshold of invariant tori.
In problems with a well-defined physical meaning, one can compare the KAM results with a measure of the parameter(s).
For example in the restricted 3-body problem, one aims to prove the theorem for the true value of the
mass ratio of the primaries. If we consider an asteroid under the gravitational attraction of
Jupiter and the Sun, then the mass ratio amounts to $\varepsilon\simeq 10^{-3}$,
which represents the benchmark that one wants to reach through rigorous KAM estimates.

Model problems like the standard maps do not have a physical reference value; therefore, one needs to apply numerical techniques
that allow to determine the KAM breakdown threshold. Among the others, we mention Greene's technique (\cite{Greene79}),
frequency analysis (\cite{LaskarFC92}), Sobolev's method (\cite{CallejaC10}).

In the next Sections we review two methods for the numerical
computation of the breakdown threshold that have been successfully
applied to the standard map (\cite{Greene79,CallejaL10,CallejaC10}):
one is based on Sobolev's method (Section~\ref{sec:sobolev}) and the
other is based on Greene's method (Section~\ref{sec:greene}).
The problem of breakdown of KAM tori has been studied by many methods.
The paper \cite{CallejaL10} contains a small survery and comparison
of several  different methods, some of which we will not mention here.

\subsection{Sobolev breakdown criterion}\label{sec:sobolev}

To illustrate the method, we focus on the specific examples of the conservative
and dissipative standard maps;
hence we have a two-dimensional discrete system, which can be parametrized by
a one-dimensional variable $\theta\in\torus$. In particular, in the conservative case we write
the invariance equation for $K$ as
$$
f\circ K(\theta)=K(\theta+\omega)\ ,
$$
while in the dissipative case we write the invariance equation for $(K,\mu)$ as
\begin{equation}\label{invariance}
f_{\mu}\circ K(\theta)=K(\theta+\omega)\ .
\end{equation}
As shown rigorously in \cite{CallejaL10} for the conservative case and in \cite{CallejaC10} for
the dissipative case, the continuation method based in
the constructive Newton method can (if given enough computer resources)
reach arbitrarily close to the breakdown. Furthermore,
the breakdown of analytic tori happens if and only if some Sobolev
norm of sufficiently high order blows up.

This rigorous result can, of course, be readily implemented.  Today's
computers, of course, do not have infinite resources, but they
are fairly impressive for people who cut their teeth in a PDP-11
with 16K of RAM. Since the algorithms we describe are based on
computing Fourier series,
one can get readily the Sobolev norms  of
the embedding $K$ and monitor their blow up.

The blow up of
the Sobolev norm, gives a clear indication that the torus
is breaking down. Note that, given the a-posteriori theorem,
and the bootstrap of regularity results, if the norm of
the computed solution is not blowing up, it is a very clear
indication that the torus is there.

\begin{remark}
Something that increases the possible effectiveness
of this method is that it has been found empirically that the
blow up of Sobolev norms is given by power laws whose
exponents are universal. Even if this is mainly an
empirical observation (that needs to be somehow tone down
since \cite{CallejaL10} contains several warnings for some maps),
it can improve dramatically the computation of breakdowns.
Many of these empirical results are organized using
\emph{Renormalization Group  methods} (\cite{Rand92a,Rand92b,Rand92c,McKay82}).
Even if some aspects of renormalization group have
been made rigorous
(\cite{Koch99,Koch04,Koch08,Koch16,Stirnemann93,Stirnemann97}), much more mathematical
work seems to remain.
\end{remark}

We implement the method for the conservative and dissipative standard maps,
computing in Table~\ref{tab:sobolev} the value of $\varepsilon_{crit}$ for
the frequency equal to the golden ratio: $\omega= 2\pi{{\sqrt{5}-1}\over 2}$.
The result in the conservative case is in full agreement with the value which can
be obtained by implementing Greene's method (see \cite{Greene79}).
The values for the dissipative case given in Table~\ref{tab:sobolev}
will be compared in Section~\ref{sec:greene} to those obtained implementing
a version of Greene's method for the dissipative standard map.

\begin{table}[h]
\centering
\begin{tabular}{|c||c|c|c|c|c|}
\hline
Conservative case&  Dissipative case &\\
\hline
$\varepsilon_{crit}$&$\lambda$ & $\varepsilon_{crit}$ \\
\hline
0.9716&0.9 &
0.9721\\
\hline &0.5 &
0.9792 \\
\hline
\end{tabular}
\vskip.1in
\caption{Breakdown values of the golden mean curve obtained implementing
Sobolev's method for the conservative case (left column) and for the dissipative case (right column),
the latter one for two different values of the dissipative parameter.}\label{tab:sobolev}
\end{table}

In Figure~\ref{Bat}, we present the existence domain of the dissipative standard map \equ{DSM} with a two harmonic potential given by
\beqa{twoharmonic}
V(\theta) = \eps_1 \sin(x) +\eps_2 \sin(2 x)\ .
\eeqa
\begin{figure}[h!]
  \centering
  \includegraphics[width=6cm,height=4cm]{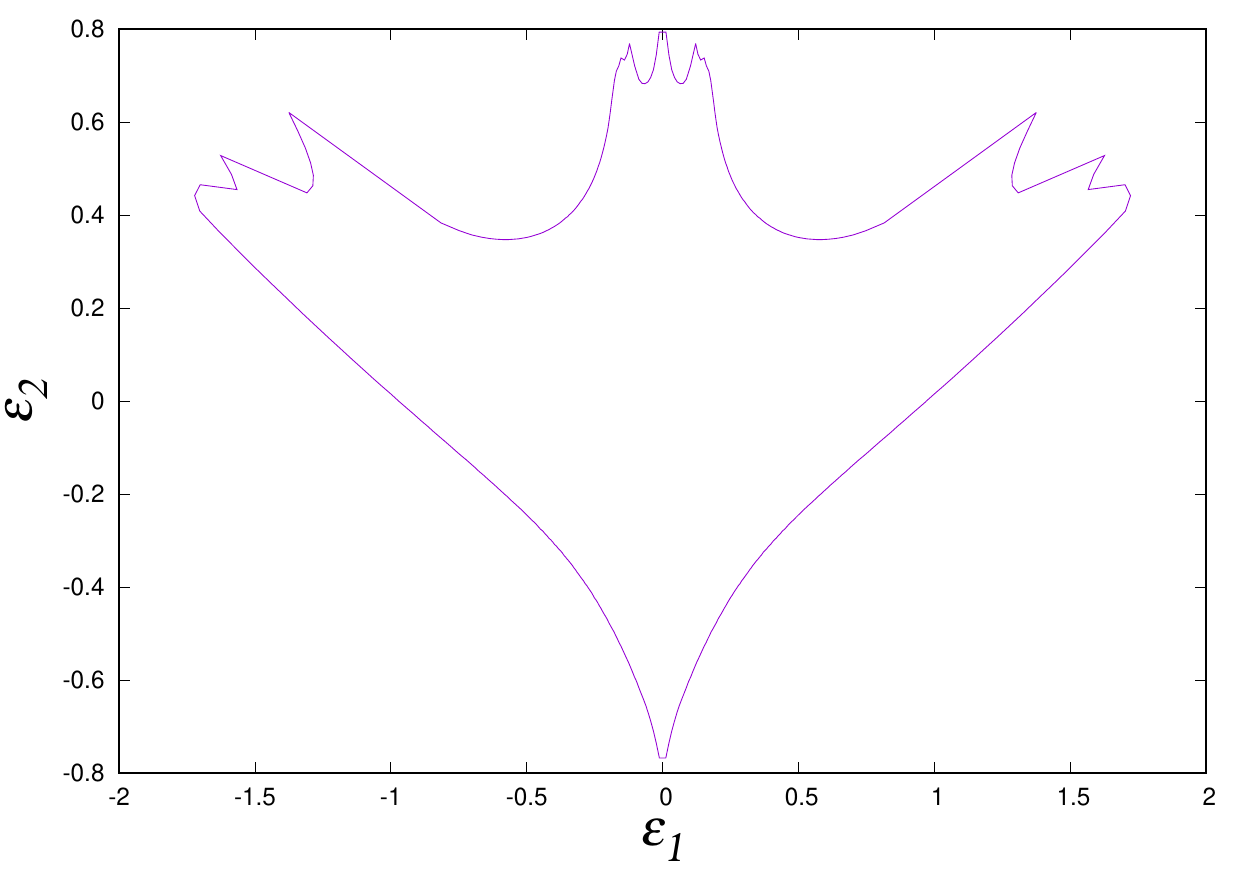}
  \includegraphics[width=6cm,height=4cm]{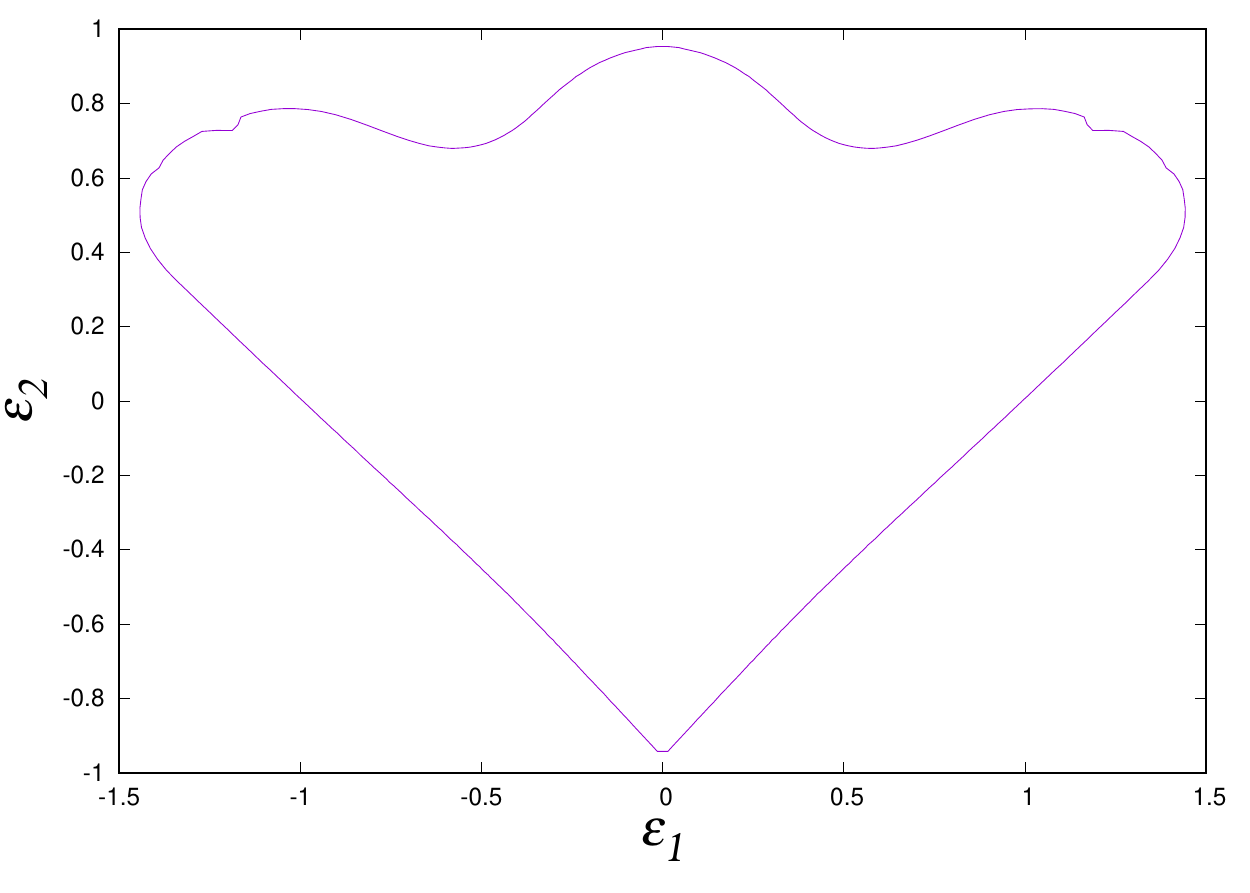}
\caption{Existence domain for invariant circles of the dissipative standard map with
  potential \equ{twoharmonic}. Left: $\lambda=0.9$. Right: $\lambda=0.1$.}\label{Bat}
\end{figure}

We call attention to the fact that this region contains parts with smooth
boundaries, but -- specially in the conservative case -- it contains
some parts of the boundary that are rather ragged. A tentative explanation
(\cite{McKay92}) is that  the smooth parts of the
the boundary of the region of existence are the intersection of the
family considered with the stable manifold of  fixed point of
renormalization. Even if this is not a completely rigorous picture,
there has been significant mathematical progress in verifying it
in an open set of families. We hope that,
in the future there could be more progress in this area.

One important advantage of the  Sobolev method is that it can be programmed
systematically and run unattended. The Greene's method relies on
periodic orbits and one has to pay attention to making sure that the
periodic orbits are continued correctly. We also note that the
Sobolev method works for models of long range interaction in
Statistical Mechanics without a dynamical interpretation.

\subsection{Greene's method, periodic orbits and Arnold's tongues}\label{sec:greene}

The method developed by J. Greene in \cite{Greene79} is based on the conjecture that
the breakdown of an invariant curve with frequency $\omega$, say $\C(\omega)$, is related to
a change from stability to instability of the periodic orbits $\PP({{p_j}\over {q_j}})$
with frequencies ${{p_j}\over {q_j}}$ tending to $\omega$. We observe that a standard
procedure to obtain the rational approximants of $\omega$ is to compute the successive
truncations of the continued fraction representation of $\omega$.

Greene's method has been successfully developed for the conservative standard map for which a partial justification is given in
\cite{FalcoliniL92,McKay92}. In the dissipative case, there appears an extra difficulty due to the fact that the periodic orbits
with frequency ${{p_j}\over {q_j}}$ occur in a whole interval of the drift parameter.
This phenomenon gives rise to the appearance of the so-called \sl Arnold tongues. \rm
Figure~\ref{arntongues}, left panel, gives a graphical
representation of the Arnold tongues; having fixed a value of the dissipative parameter
$\varepsilon$, there is a whole interval of the drift parameter $\mu$
which admits a periodic orbit of the same period. The right panel of Figure~\ref{arntongues}
shows several periodic orbits approaching the torus with frequency equal to the golden
mean; such periodic orbits have frequency equal to the rational approximants which
are given by the ratio of the Fibonacci numbers.

A partial justification of an extension of Greene's criterion in the conformally symplectic case is presented in \cite{CCFL14},
where it is proved that if there exists a smooth invariant attractor, one can predict the eigenvalues of
the periodic orbits approximating the torus for parameters close to those of the attractor.

\begin{figure}[ht]
\centering
\includegraphics[width=6truecm,height=5truecm]{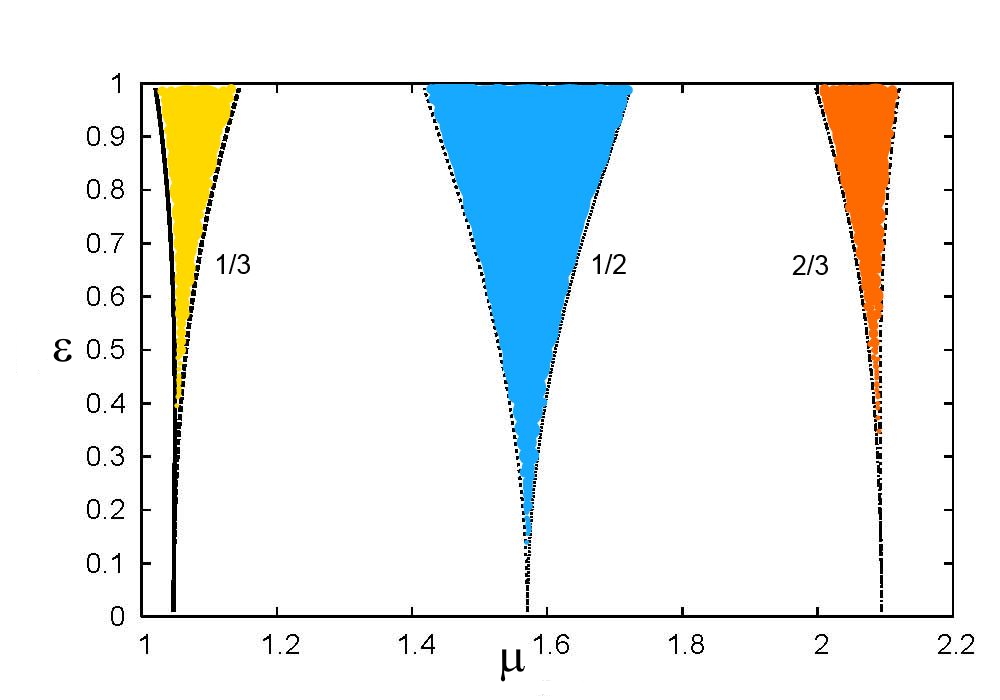}
\includegraphics[width=6cm,height=4.5cm]{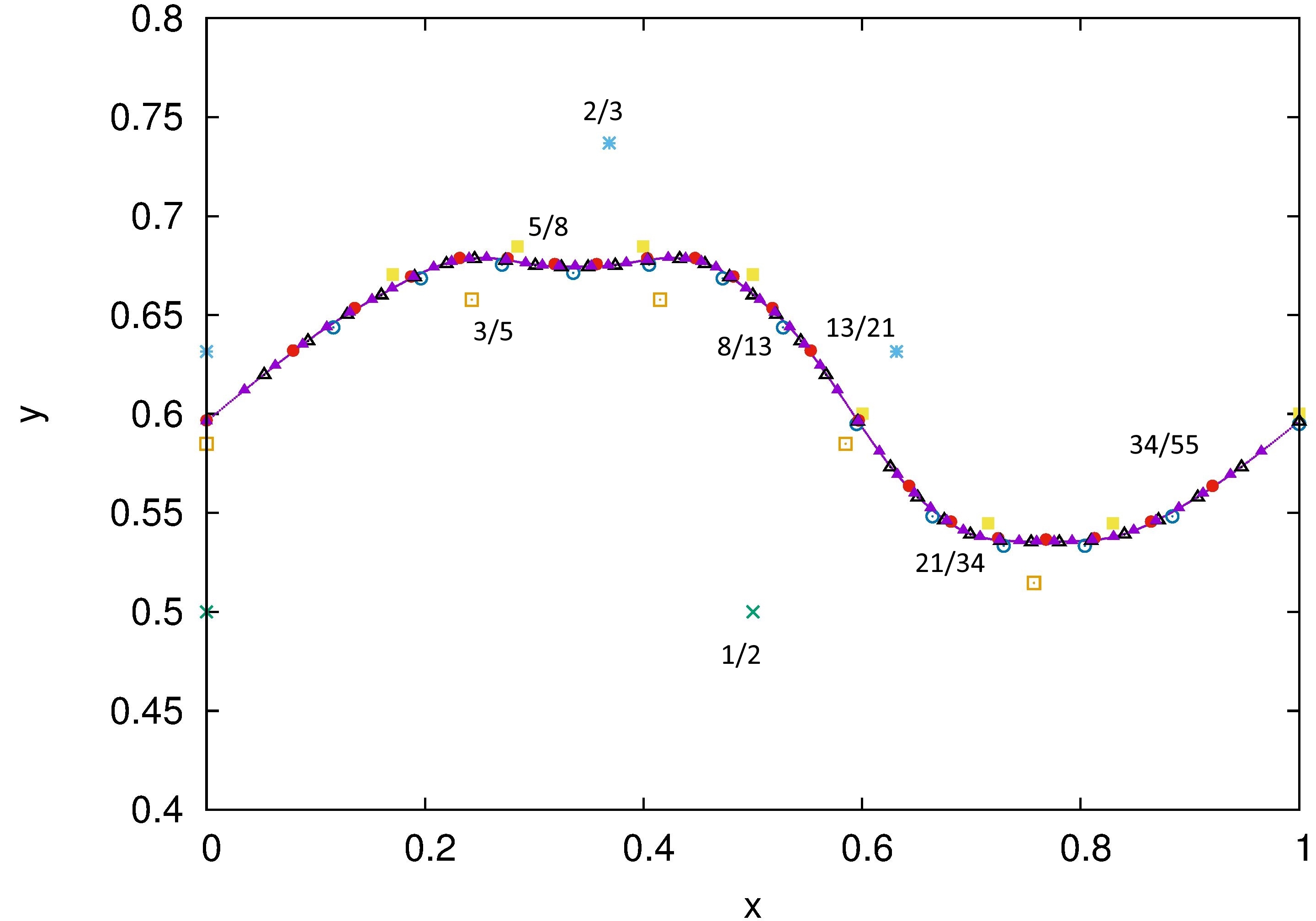}
\caption{Left: Arnold's tongues providing $\mu$ vs. $\varepsilon$ for three periodic orbits of
the dissipative standard map with periods 1/3, 1/2, 2/3. Right: periodic orbits of the dissipative standard map
approximating the golden mean curve.} \label{arntongues}
\end{figure}

Figure~\ref{drift} shows some approximating periodic orbits (left panels) and the corresponding behaviour
of the drift parameter (right panels) that, in the limit, tends to the value of the drift that
corresponds to the golden mean torus.

\begin{figure}
\centering
\includegraphics[width=4.9cm,height=5cm]{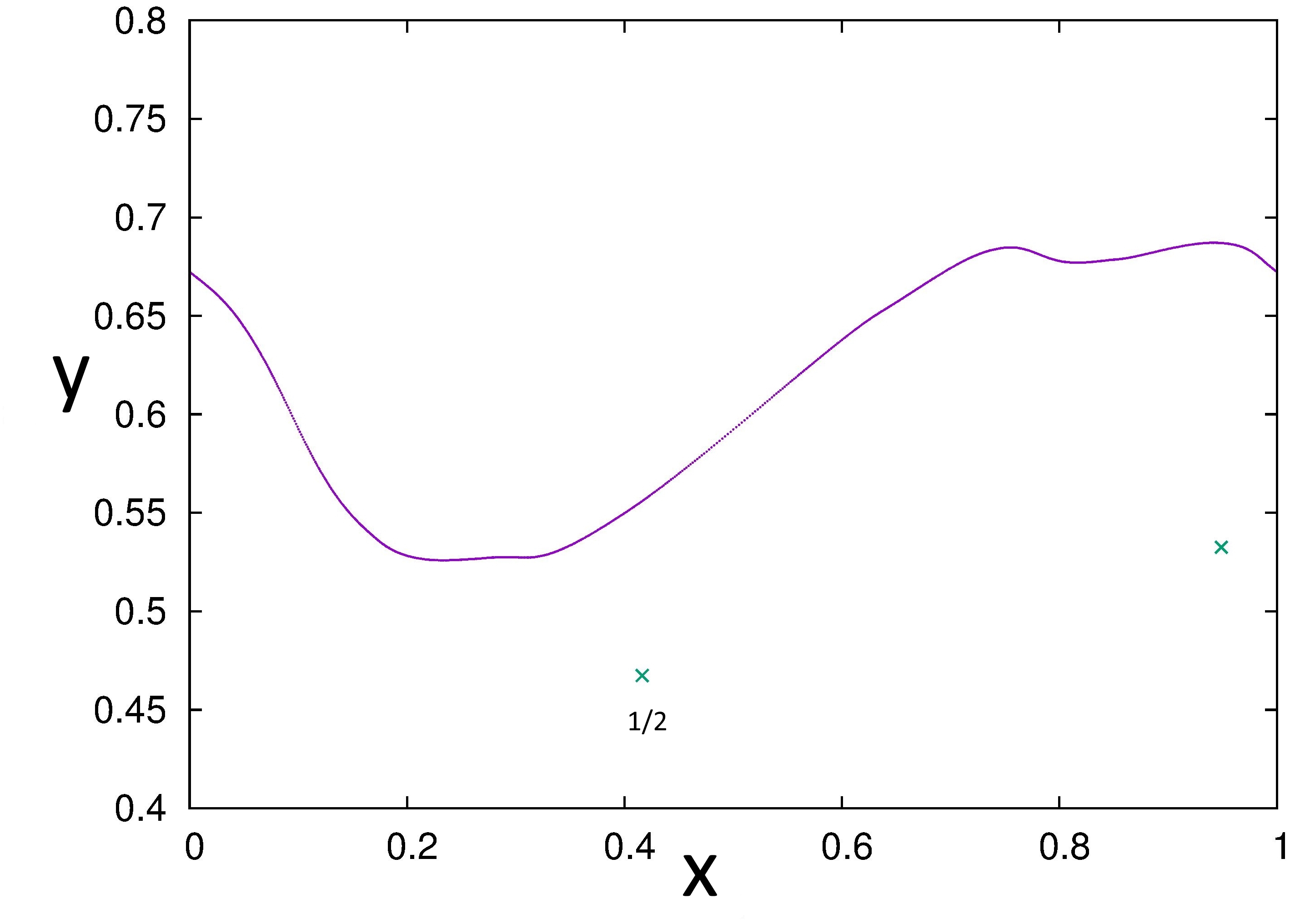}
\includegraphics[width=5.5cm,height=5.5cm]{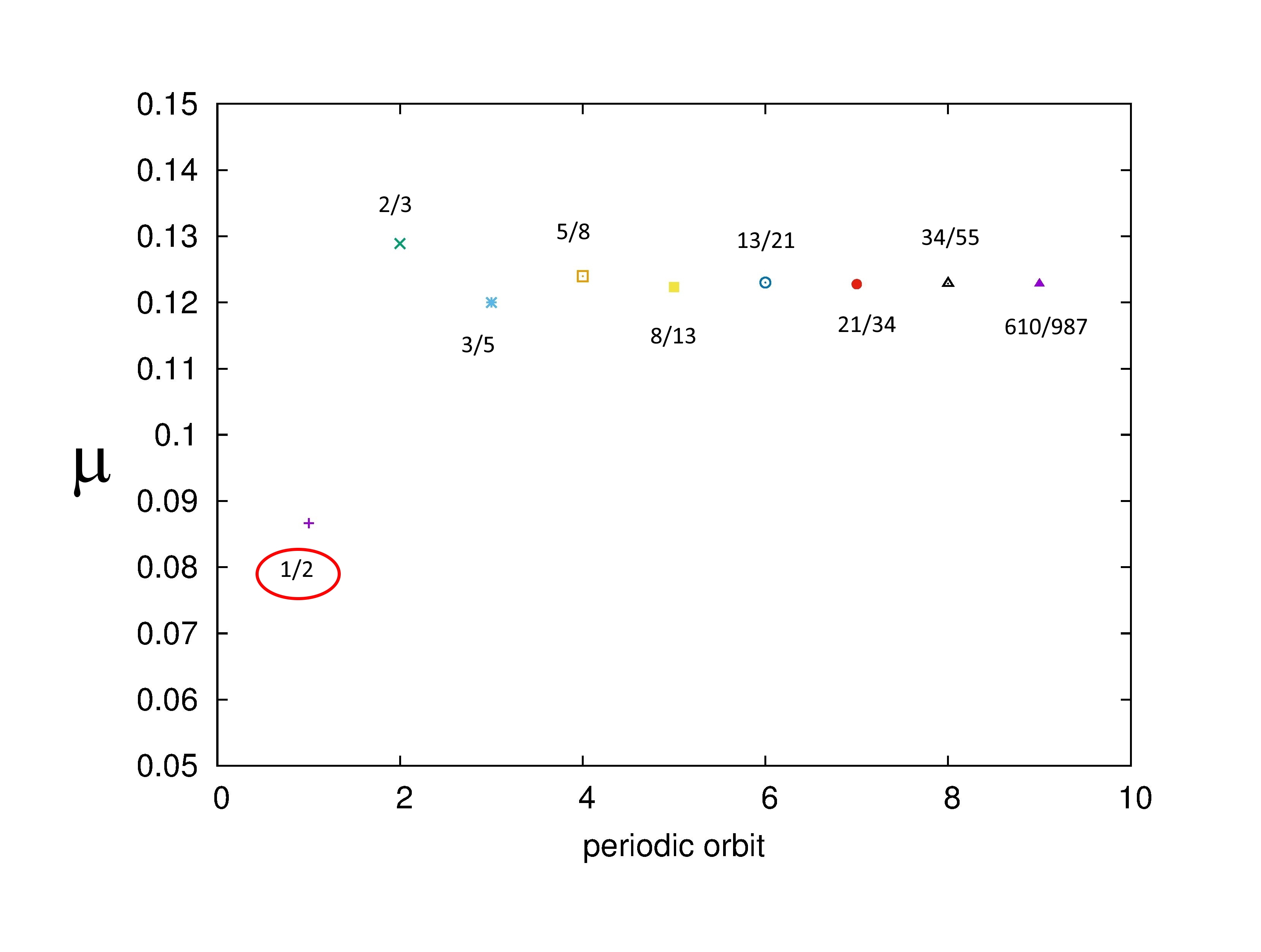}\\
\includegraphics[width=5cm,height=5cm]{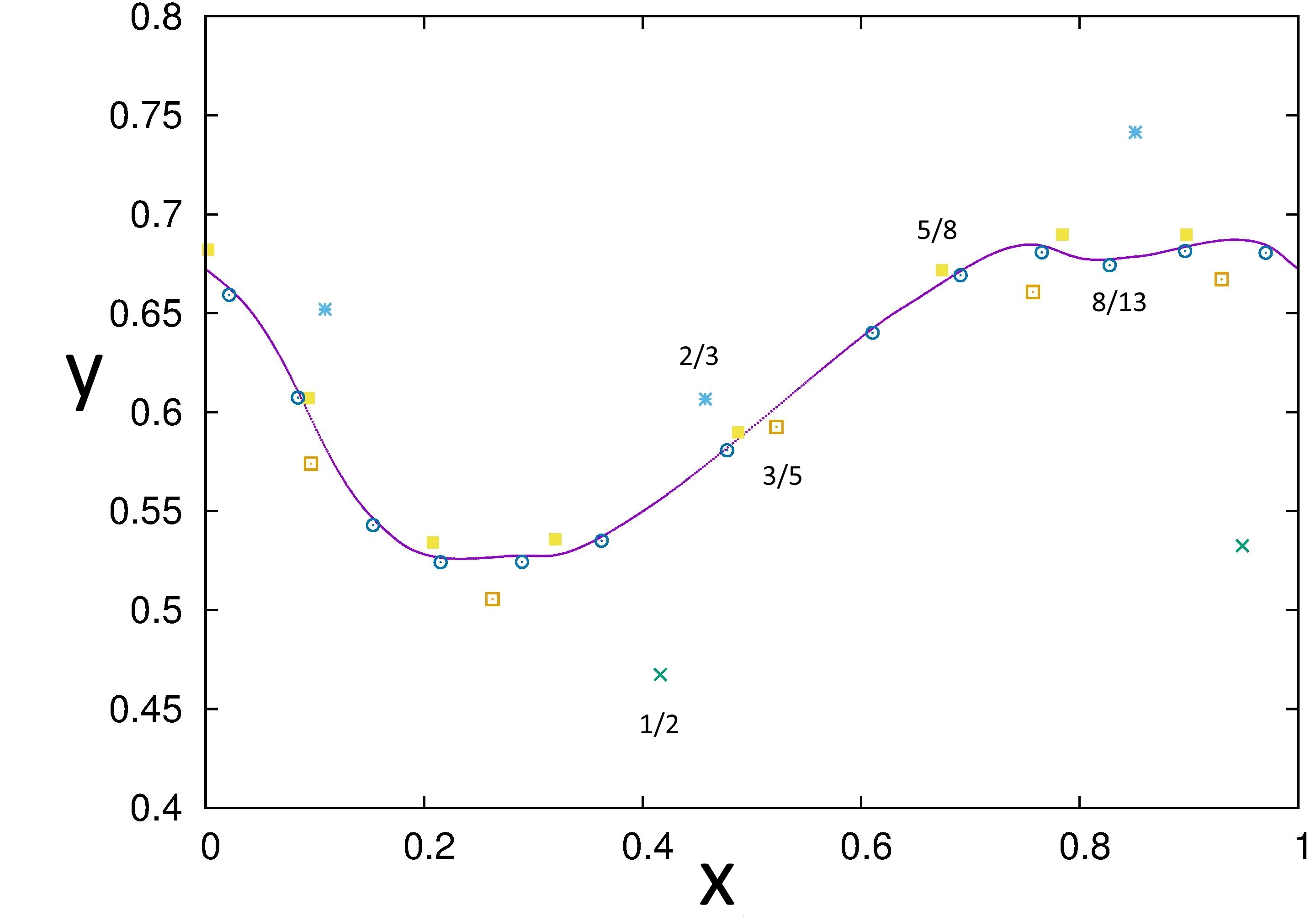}
\includegraphics[width=5.5cm,height=5.5cm]{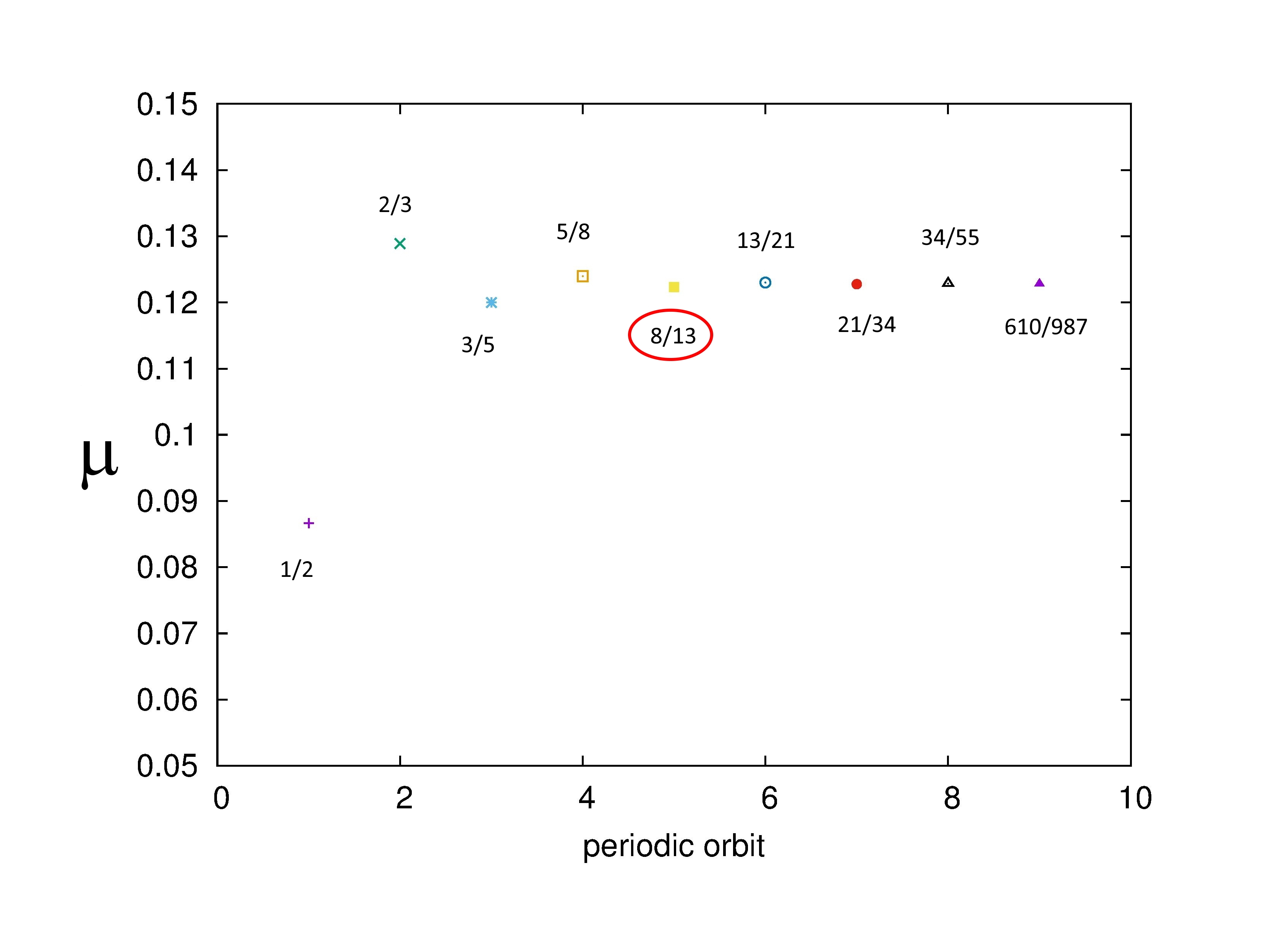}\\
\includegraphics[width=5cm,height=5cm]{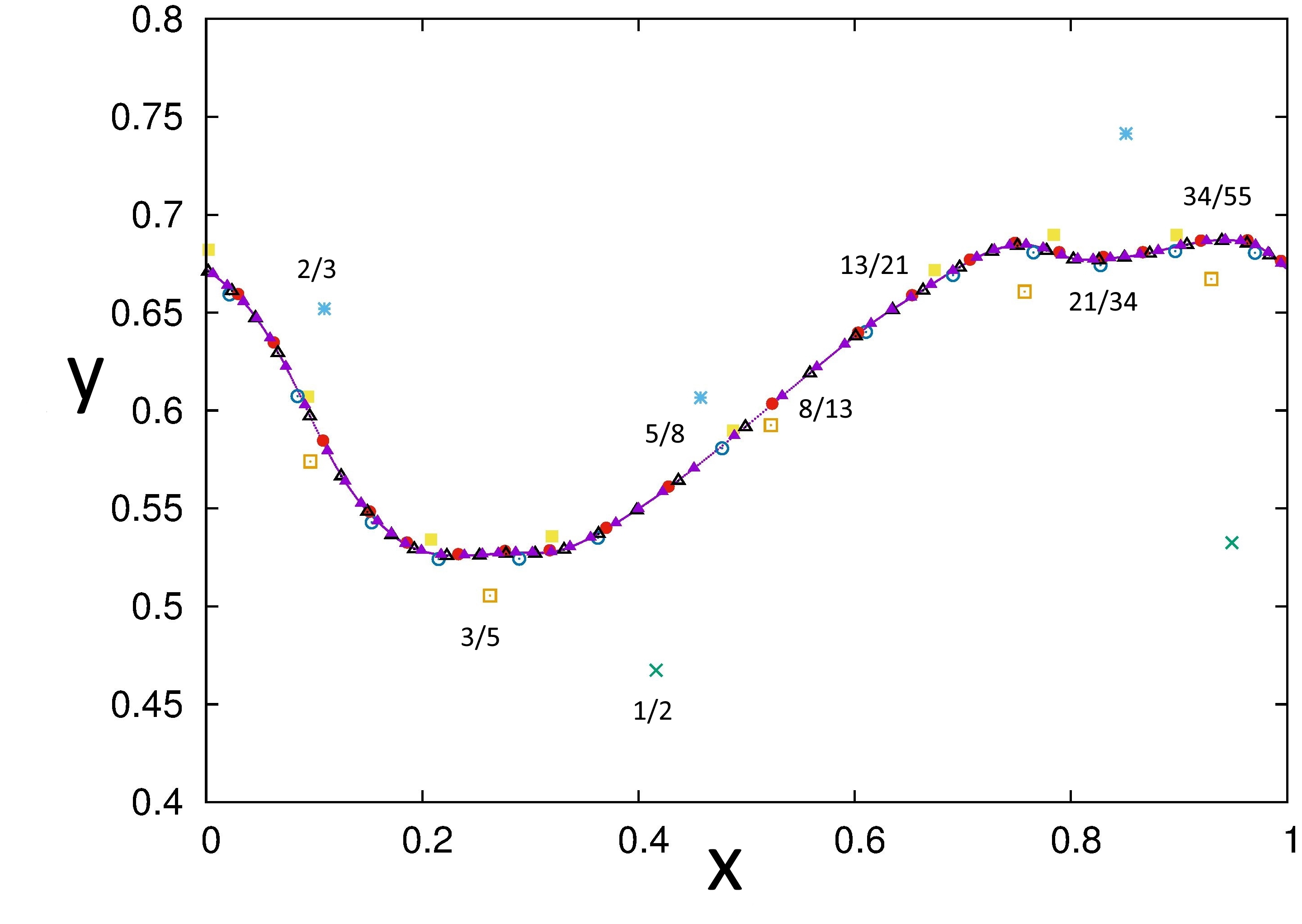}
\includegraphics[width=5.5cm,height=5.5cm]{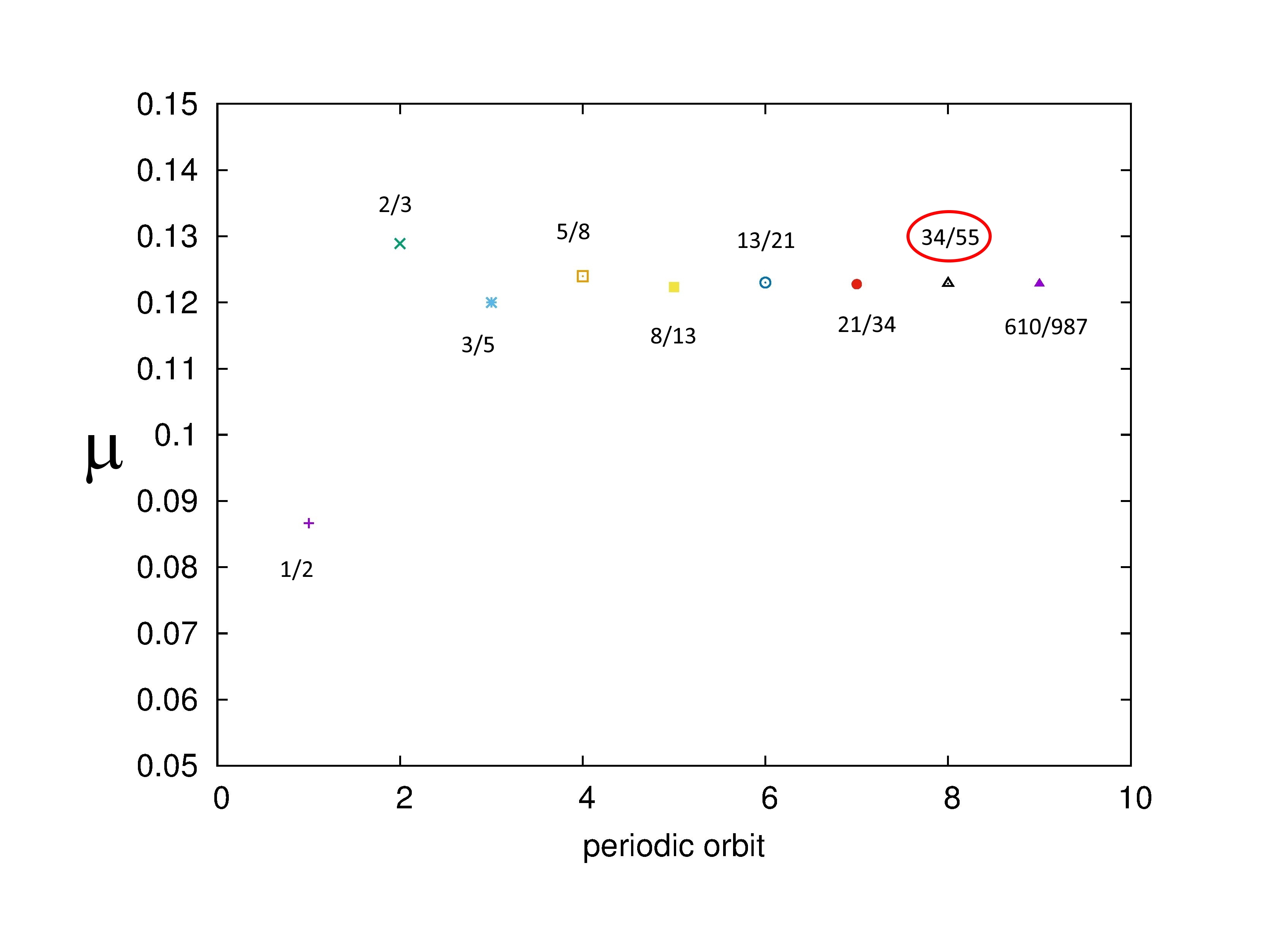}
\caption{Left: periodic orbits with increasing periods, approximating the golden mean curve.
Right: the corresponding drift parameters with the successive periodic orbits labeled by
integer numbers on the $x$-axis.} \label{drift}
\end{figure}

We also call attention to \cite{CellettiMK07} which contains
tentative results on the non-existence of invariant tori for the spin-orbit models.
Even if the methods developed there are not rigorous, they may
present a counterpoint to the methods to study the existence.

\section{Collision of invariant bundles of quasi-periodic attractors}\label{sec:collision}

Quasi-periodic attractors
of conformally symplectic maps are normally hyperbolic invariant manifolds (NHIM).
We can obtain the Lyapunov multipliers of the attractor from a simple computation.
We start from
the invariance equation \equ{invariance} for a pair $(K, \mu)$. We then introduce a change
of variables to reduce the cocycle. Let
$\tilde M({\underline \theta})=[DK({\underline \theta})\,|\,E^s({\underline \theta})]$
be the matrix whose columns
are the tangent and stable bundles of $\mathcal{K} = K(\torus^n)$:
\begin{equation}\label{reducibility}
Df_\mu \circ K ({\underline \theta}) \tilde M({\underline \theta}) = \tilde M({\underline \theta} + {\underline \omega})
\left(\begin{array}{cc} 1 & 0\\
0 & \lambda \\ \end{array}\right)\ .
\end{equation}
From equation \equ{reducibility} we can write the stable bundle as follows
$$
E^s({\underline \theta}) = DK({\underline \theta}) B({\underline \theta})
+ J^{-1}DK({\underline \theta})N({\underline \theta})\ ,
$$
where $B({\underline \theta})$ is the function that satisfies
$$
B({\underline \theta}) - \lambda B({\underline \theta} + {\underline \omega}) = -S({\underline \theta})\ .
$$
Indeed, after $j$ iterates of the map we have that,
$$
Df_\mu^j \circ K ({\underline \theta}) = \tilde M({\underline \theta} + {\underline \omega})
\left(\begin{array}{cc} 1 & 0\\
0 & \lambda^j \\ \end{array}\right)\ \tilde M^{-1}({\underline \theta}) \ ,
$$
which shows that the tangent space of $\mathcal M$ at $K(\theta)$ is
$$
T_{K(\theta)} \mathcal M = \mathrm{Range }(DK(\theta)) \oplus E_{K(\theta)}^s\ .
$$
We can conclude that there exists a constant $C$ such that
\beqano
C^{-1}\lambda^j |v|\leq |Df_\mu^j \circ K ({\underline \theta})\ v|\leq C\lambda^j|v|\ ,\qquad
v\in E^s_{K({\underline \theta})}\ ,\nonumber\\
C^{-1}\ |v|\leq |Df_\mu^j \circ K ({\underline \theta})\ v|\leq C\ |v|\ ,\qquad\
v\in E^c_{K({\underline \theta})}\ , \eeqano showing that $\mathcal
K=K(\torus^n)$ is a NHIM. Equation \equ{reducibility} also tells us that the
Lyapunov multipliers are constant
along the family of quasi-periodic attractors for fixed Diophantine vectors.

In the case of maps of the cylinder $\mathcal{M} = \mathbb R \times \mathbb T$,
we know that the curve $\mathcal K$ is $C^r$, one dimensional,
and since
$\omega$ satisfies the Diophantine condition, we know by the results of
\cite{Herman79,SinaiK87,KatznelsonO89a,KatznelsonO89b} that the
map conjugating the dynamics in $\mathcal K$ to a rigid rotation
is in $C^{r-\tau-\delta}$ for a small $\delta > 0$.
Therefore, by the bootstrap
of regularity results\footnote{i.e., all tori
  which are smooth enough are analytic if the map is analytic (\cite{CallejaCL11}).},
the conjugacy is analytic for analytic maps. Since the bundles depend on
the conjugacy, then the
regularity of the manifold implies the analyticity of $K$ and
the bundles up to the breakdown.

To investigate the breakdown of
normal hyperbolicity, we note that, because of the pairing rule of Lyapunov
exponents \cite{DettmannM96,WojtkowskiL98}, since one Lyapunov multiplier
is $1$ -- the one along the tangent directional (remember that the
map on the torus is smoothly conjugate to the torus) -- the other one is
precisely $\lambda$.

We recall that hyperbolicity is equivalent to the  existence of
\sl transversal \rm invariant bundles with different rates.
In our case, if the tori have to cease to be normally hyperbolic,
because the exponents remain constant, the only thing that can
happen is that the transversality of the bundles deteriorates.

What is found empirically is that the breakdown happens
because at the same time the regularity of the conjugacy
deteriorates quantitatively (even if the conjugacy remains
analytic, some Sobolev norm blows up).

At the same parameter values,
the breakdown of hyperbolicity happens via
the stable and tangent bundle collision.
Even if the  Lyapunov exponents remain safely away, the transversality
deteriorates and the tangent and stable bundles become close to
tangent.

In the case at hand, we can make a very detailed study:
the bundles are
one dimensional and we  compute a formula for the angle between the bundles
for every $\theta$. In fact, let $\alpha(\theta)$ be the angle between the stable and
tangent bundles for every $\theta \in \torus$, then we have
$$
\alpha(\theta) = \arctan\left( \frac{1}{B(\theta)(DK(\theta)^T DK(\theta))} \right).
$$
This formula says that the angle $\alpha(\theta)$ goes to zero at points where the functions
in the denominator go to infinity.

We present figures (see Figure~\ref{figdrif2}) of the angle between the bundles close to the breakdown.

\begin{figure}[h!]
\centering
\includegraphics[width=6truecm,height=6truecm]{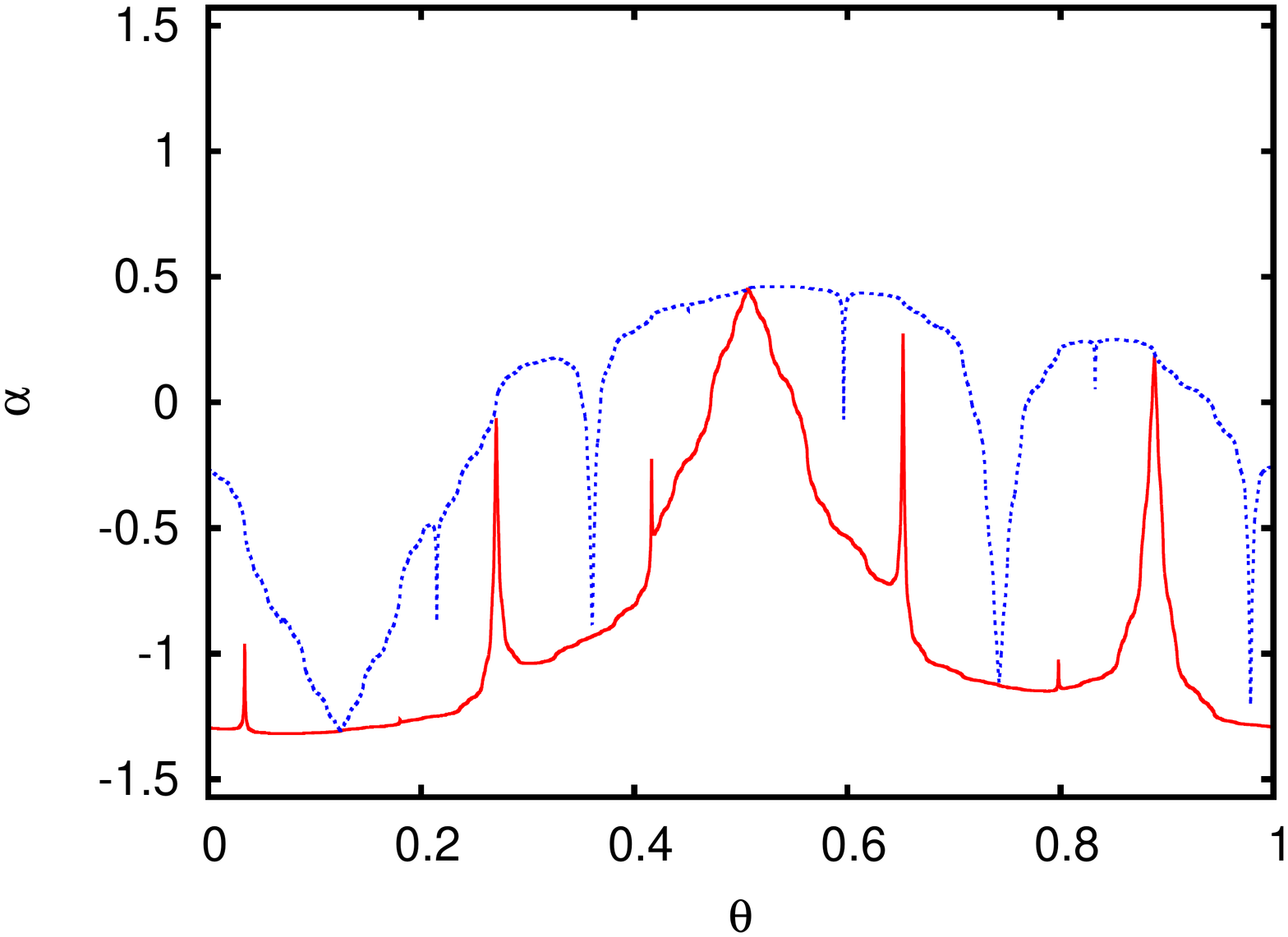}
\includegraphics[width=6truecm,height=4.0truecm]{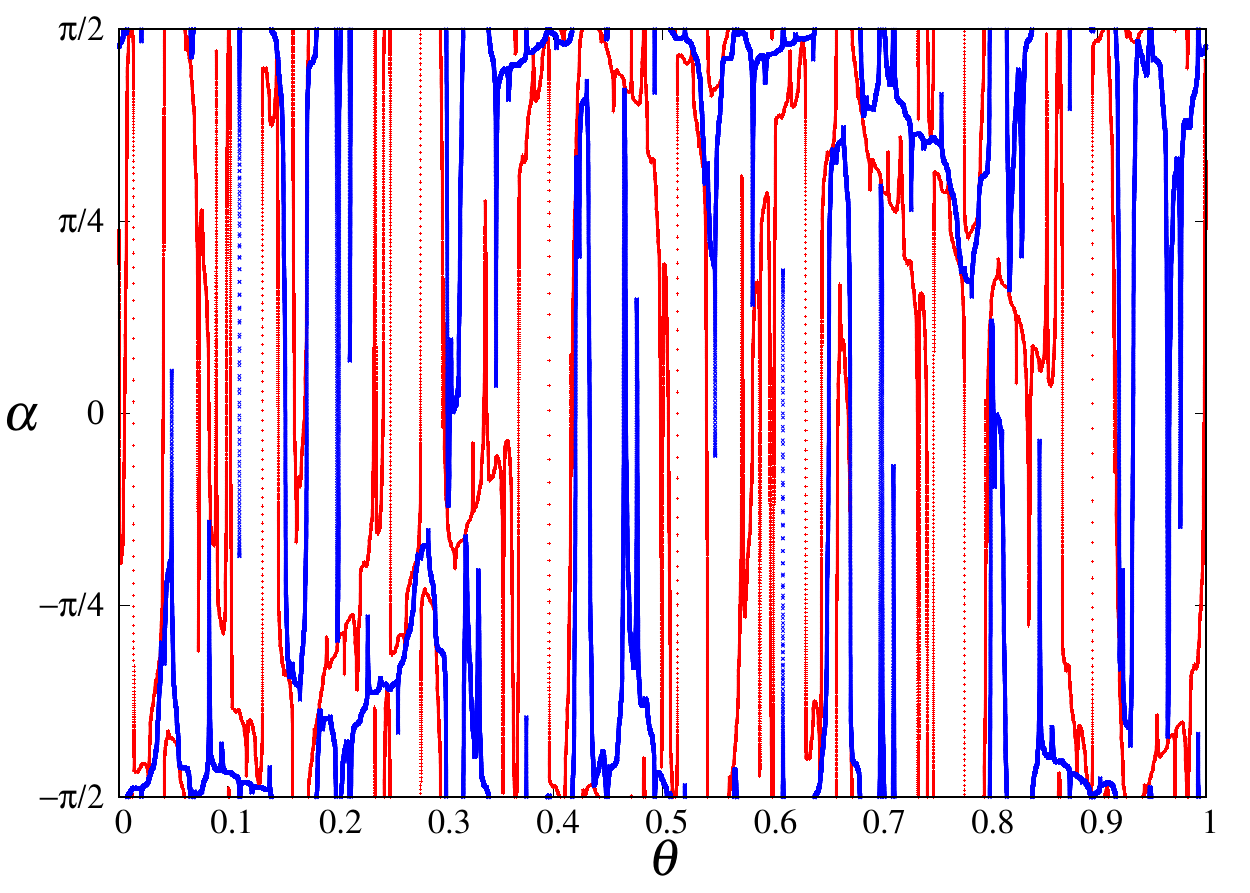}\\
\caption{Invariant bundles close to their collision. Left: dissipative standard map.
  Right: dissipative standard non-twist map. Reproduced from \cite{Cal-Can-Har-20}.}
  \label{figdrif2}
\end{figure}

Rather remarkably these two phenomena (the blow up of Sobolev norms
and the stable bundles and the tangent becoming parallel) happen
at the same time and present very unexpected regularities.
There are scaling relations that seem to be independent of
the family considered  and they happen in codimension $1$ smooth submanifolds
in the space of maps.    We think that this is a very interesting
mathematical phenomenon that deserves rigorous study. It seems
quite unlikely that it would have been discovered except for the
very careful numerics that can explore with confidence close to
the breakdown. Such delicate numerics are only possible because
of the rigorous mathematical development.

\section{Applications}\label{sec:applications}

In this Section we want to briefly review some applications of KAM theory for conservative and dissipative models.
We will consider applications to the standard map and to the spin-orbit problem, both in the conservative and dissipative
settings. Although we will not present other applications of KAM theory,
it is worth mentioning also the constructive KAM results to the N-body and planetary problems in Celestial Mechanics (\cite{Poincare});
in this context,
for results obtained in the conservative framework we refer the reader to \cite{Arnold63b,CellettiC97,MR2267954,CellettiC07,LGS1,LGS2}
and to \cite{MR2777134} for numerical investigations including dissipative effects.

\subsection{Applications to the standard maps}
The first applications of computer-assisted KAM proofs have been given for the conservative standard map;
these results show that the golden mean torus persists for values of the perturbing parameter equal to
$93\%$ of the numerical breakdown value (see \cite{LlaveR90,LlaveR90b}); we also mention \cite{CellettiC95} which,
at the same epoch but using a different approach than \cite{LlaveR90,LlaveR90b}, reached $86\%$ of the numerical breakdown value.

Rigorous estimates for the conservative standard map using the a-posteriori method have been proved in the remarkable paper \cite{FiguerasHL17}, where for
the twist and non-twist conservative standard maps the golden mean torus is proved to persist
for values of the perturbing parameter as high as 99.9\% of the numerical breakdown value.

For the dissipative standard map, the paper \cite{CCL2020} analyzes the persistence of the invariant
attractor with frequency equal to the golden mean and for a fixed value of the dissipative parameter (precisely
$\lambda=0.9$); such persistence is shown for values of the perturbing parameter equal to 99.9\% of the breakdown value, where the numerical value
has been obtained through the techniques presented in Sections~\ref{sec:sobolev} and \ref{sec:greene}.

\subsection{Applications to the spin--orbit problems}
The first application of KAM theory to the conservative spin-orbit problem is found in \cite{Celletti90I,Celletti90II}.
In those articles some satellites in synchronous spin-orbit resonance have been considered; the synchronous or
1:1 spin-orbit resonance implies that the satellite always points the same face to the host planet.
In particular, the following satellites have been considered: the Moon, and three satellites of Saturn, Rhea, Enceladus, Dione.
Being the normalized frequency (namely, the ratio between the rotational and orbital frequency) equal to one,
two Diophantine numbers bounding unity from above and below have been considered. Through a computer-assisted
KAM theorem, the existence of invariant tori with frequency equal to the bounding numbers have been established
for the true values of the parameters of the satellites, namely the eccentricity and the equatorial oblateness.

\vskip.1in

Such result guarantees the stability for infinite times in the sense of confinement in the phase space.
In fact, the phase space associated to the Hamiltonian describing the conservative spin-orbit problem
is 3-dimensional; since the KAM tori are 2-dimensional, one gets a confinement of the motion
between the bounding invariant tori.

We remark that the confinement is no more valid for $n>2$ degrees of freedom, since the motion can diffuse through invariant tori,
reaching arbitrarily far regions; this phenomenon is known as Arnold's diffusion (\cite{Arnold64}) for which we refer
to the extensive literature on this topic (see, e.g., \cite{DelshamsLS2006,GideaLS2020} and references therein).

\vskip.1in

For the dissipative spin-orbit problem, we refer to \cite{CellettiC2009} for the development
of KAM theory for a model of spin-orbit interaction with tidal torque as in \equ{SOeqdiss}.
Precisely, for $\lambda_0\in{\real_+}$ and $\omega$ Diophantine,
it is proven that there exists $0<\varepsilon_0<1$, such that
for any $\varepsilon\in[0,\varepsilon_0]$ and any $\lambda\in[-\lambda_0,\lambda_0]$ there exists
a unique function $K=K(\theta,t)$ and a drift term $\mu$ which is the solution of the invariance equation
for the dissipative spin-orbit model.

Explicit estimates for the dissipative spin-orbit problem, even in the more general case
with a time-dependent tidal torque as in \equ{sodiss1}, are given in \cite{CCGL2020} (see also \cite{Locatelli}). Here, the a-posteriori
method is implemented to construct invariant attractors with Diophantine frequency; the results are valid
for values of the perturbing
parameter consistent with the astronomical values of the Moon and very close to the numerical breakdown
threshold, which has been computed in \cite{CCGL2020} through Sobolev and Greene's method (see also \cite{StefanelliL15}).

\vglue1cm

\noindent
{\bf Acknowledgements.} R.C. was partially supported by DGAPA-UNAM Project IN 101020.
A.C. acknowledges the MIUR Excellence
Department Project awarded to the Department of Mathematics,
University of Rome Tor Vergata, CUP E83C18000100006, EU-ITN Stardust-R,
MIUR-PRIN 20178CJA2B ``New Frontiers of Celestial Mechanics: theory and
Applications''.
R.L was partially supported by NSF grant DMS 1800241.


\def\cprime{$'$} \def\cprime{$'$} \def\cprime{$'$} \def\cprime{$'$}

\end{document}